\documentclass[pdflatex,sn-mathphys-num]{sn-jnl}

\usepackage{graphicx}%
\usepackage{multirow}%
\usepackage{amsmath,amssymb,amsfonts}%
\usepackage{amsthm}%
\usepackage{mathrsfs}%
\usepackage[title]{appendix}%
\usepackage{xcolor}%
\usepackage{textcomp}%
\usepackage{manyfoot}%
\usepackage{booktabs}%
\usepackage{algorithm}%
\usepackage{algorithmicx}%
\usepackage{algpseudocode}%
\usepackage{listings}%

\usepackage[UKenglish]{babel}
\usepackage{graphicx}
\usepackage{subfigure}
	\graphicspath{{./}{figures/}}
\usepackage{amssymb,empheq,bbold}
\usepackage[inline]{enumitem} 
\usepackage{geometry}
\usepackage{caption}
\usepackage{amsthm}
\usepackage{bbm}
\usepackage{enumerate}
\usepackage{multirow}

\theoremstyle{thmstyleone}%
\newtheorem{theorem}{Theorem}
\newtheorem{prop}[theorem]{Proposition}%

\theoremstyle{thmstyletwo}%
\newtheorem{remark}{Remark}%

\theoremstyle{thmstylethree}%
\newtheorem{definition}{Definition}%

\DeclareMathOperator*{\esssup}{ess\,sup}
\newtheorem{lemma}{Lemma}

\begin{document}

\title[Coupling local and nonlocal total variation flow for image despeckling]{Coupling local and nonlocal total variation flow for image despeckling}

\author[1]{\fnm{Yi} \sur{Ran}}
\author[1]{\fnm{Zhichang} \sur{Guo}}
\author[2]{\fnm{Kehan} \sur{Shi}}
\author*[1]{\fnm{Qirui} \sur{Zhou}}\email{21b912020@stu.hit.edu.cn}
\author[3]{\fnm{Jingfeng} \sur{Shao}}
\author[4,5]{\fnm{Martin} \sur{Burger}}
\author[1]{\fnm{Boying} \sur{Wu}}

\affil*[1]{School of Mathematics, Harbin Institute of Technology, 150001 Harbin, China}
\affil[2]{Department of Mathematics, China Jiliang University, 310018 Hangzhou, China}
\affil[3]{School of Mathematics and Information Science, Guangxi University, 530004 Guangxi, China}
\affil[4]{Helmholtz Imaging, Deutsches Elektronen-Synchrotron DESY, 22607 Hamburg, Germany}
\affil[5]{Fachbereich Mathematik, Universit{\"a}t Hamburg, 20146 Hamburg, Germany}





\abstract{
    Nonlocal equations effectively preserve textures but exhibit weak regularization effects in image denoising, whereas local equations offer strong denoising capabilities yet fail to protect  textures.	
    To integrate the advantages of both approaches, this paper investigates a coupled local-nonlocal total variation flow for image despeckling.		
    We establish the existence and uniqueness of the weak solution for the proposed equation.	
    Several properties, including the equivalent forms of the weak solution and its asymptotic behavior, are derived.
    Furthermore, we demonstrate that the weak solutions of the proposed equation converge to the weak solution of the classical total variation flow under kernel rescaling.	
    The importance of coupling is highlighted through comparisons with 
    local and nonlocal models for image despeckling. 
}


\maketitle

\section{Introduction}

Synthetic aperture radar (SAR) is a coherent imaging system that is widely utilized in landscape classification, disaster monitoring, and surface change detection \cite{perera2023sar}. 
However, SAR images are frequently affected by signal-dependent and spatially correlated noise, known as speckle, which significantly degrades image quality.
The degeneration can be mathematically represented \cite{goodman1975statistical} as 
$$
f=\tilde{u} \eta,
$$
where $\tilde{u}$ denotes the clean image, $f$ is the observed image and $\eta$ represents the multiplicative Gamma noise following a Gamma distribution with mean $1$.
The probability density function related to the multiplicative Gamma noise is given by
$$
p(\eta)=\frac{L^{L}}{\Gamma(L)} \eta^{L-1} e^{-L \eta} \mathbf{1}_{\{\eta \geqslant 0\}},
$$
where $\mathbf{1}_{\{\eta \geqslant 0\}}$ is the indicator function of $\{\eta \geqslant 0 \}$ and $L$ represents the number of looks.

To remove multiplicative Gamma noise, the well-known variational model (AA model) \cite{aubert2008variational} 
\begin{equation*}
    \inf_{u\in BV(\Omega), u>0}E(u)=\int_{\Omega}|D u|+\lambda \int_{\Omega}\left(\log u+\frac{f}{u}\right) d x,  
\end{equation*}
was introduced. 
It integrates the local total variation (TV) regularizer from the ROF model \cite{rudin1992nonlinear} with maximum a posteriori (MAP) estimation.
Its corresponding evolution equation 
\begin{equation*}
    \begin{cases}u_t=\operatorname{div}\left(\frac{D u}{|D u|}\right)+\lambda \frac{f-u}{u^2},   \quad(x, t) \in Q_T, \\ \frac{\partial u}{\partial \vec{n}}=0,  \quad(x, t) \in \partial \Omega \times(0, T), \\ u(x,0)=f(x), \quad x \in \Omega,\end{cases}
\end{equation*}
constitutes a combination of TV flow \cite{andreu2000minimizing} and a fidelity term.
Since multiplicative noise is signal-dependent, regions that only different in  grayscale values suffer different levels of noise contaminant.
To enhance the sensitivity to grayscale variations, a convex model \cite{dong2013convex} incorporating the adaptive total variation  \cite{chen2003minimization} and a novel fidelity term was proposed
\begin{equation*}
    \inf_{u\in BV(\Omega), u>0}E(u)= \int_{\Omega} g|D u|+\lambda \int_{\Omega}\left(u+f \log \frac{1}{u}\right) d x. 
\end{equation*}
The weight $g \in C(\overline{\Omega})$  with 
    $g>0 $ in the adaptive total variation is spatially dependent to control  the regularization strength across different regions.
The above TV-based methods exhibit strong denoising capabilities in homogeneous regions, but their texture restoration performance is not satisfactory due to the lack of nonlocal image information during modeling. 

Nonlocal operators have been considered to preserve texture information \cite{gilboa2009nonlocal, wen2023nonlocal}. 
Specifically, the nonlocal total variation (NLTV) regularizer is introduced in \cite{gilboa2009nonlocal}.
Based on the NLTV regularizer, numerious variational models were proposed \cite{dong2012nonlocal, karami2018nonlocal, chen2019convex}, which adapted different fidelities.

The NLTV equation
\begin{equation}\label{eq:nl_tv_general}
    \frac{\partial u}{\partial t}=\int_{\Omega} J(x-y)|u( y,t)-u(x,t)|^{-1}(u( y,t)-u(x,t)) d y,
\end{equation}
has been studied in \cite{gilboa2009nonlocal, kindermann2005deblurring}. 
When the weight function $J(x-y)$ quantifies self-similarity and redundancy within images \cite{buades2005review}, the NLTV equation can be effectively applied to image denoising, demonstrating strong performance in preserving texture and fine details.
However, the NLTV equation exhibits a weaker regularizing effect compared to local TV equation \cite{chasseigne2006asymptotic}, resulting in incomplete denoising in homogeneous regions.

In this paper, we aim to leverage the satisfactory denoising effect of local TV in homogeneous regions and nonlocal TV in textures.
Coupling local and nonlocal operators is an effective approach\cite{delon2019rnlp, shi2015salt}.
We propose the following model for image despeckling
\begin{equation}\label{eq:tv_case}
    \left\{\begin{array}{l}
    \frac{\partial u}{\partial t}=\operatorname{div}\left((1-g)\frac{D u}{|D u|} \right)+\lambda \int_{\Omega} \frac{g(x)+g(y)}{2} J(x-y)\frac{u(y, t)-u(x, t)}{|u(y, t)-u(x, t)|}d y, \quad(x, t) \in Q_T, \\
    \frac{\partial u}{\partial\vec{n}}=0, \quad(x, t) \in \partial \Omega \times(0, T), \\
    u(x, 0)=f(x), \quad x \in \Omega,
\end{array}\right.
\end{equation}
where  $\lambda>0$ represents a trade-off parameter.  
Notice that the evolution problem \eqref{eq:tv_case} is the gradient flow associated to the functional 
$$
E(u)= \int_{\Omega}(1-g)|D u| +\frac{\lambda}{2 } \int_{\Omega} \int_{\Omega} \frac{g(x)+g(y)}{2} J(x-y)|u(y)-u(x)| d y d x.
$$
When $g \equiv 1$ ($g \equiv 0$), $E(u)$ degenerates to the NLTV model \cite{gilboa2009nonlocal} (TV model \cite{andreu2000minimizing}).
In our model, 
$g \in C(\overline{\Omega})$  is spatially related to control the diffusion speed to improve the multiplicative denoising effect, which can be constructed similarly to \cite{zhou2014doubly, shan2019multiplicative}. 
More general, in areas full of textures and repeated
structures, if  we set $g(x)=1$, our model converts to the nonlocal diffusion equation
$$
\frac{\partial u}{\partial t}= {\lambda} \int_{\Omega} \frac{1+g(y)}{2} J(x-y)|u(y)-u(x)| d y.
$$
In homogeneous areas, if we set $g(x)=0$, our model leads to 
$$
\frac{\partial u}{\partial t}=\operatorname{div}\left(\frac{D u}{|D u|} \right)+\lambda \int_{\Omega} \frac{g(y)}{2} J(x-y)\frac{u(y, t)-u(x, t)}{|u(y, t)-u(x, t)|}d y. 
$$
We add TV regularization to nonlocal diffusion to improve denoising efficiency, as nonlocal diffusion has a weak regularizing effect \cite{chasseigne2006asymptotic}.
Since the evolution in \eqref{eq:tv_case} starts from the noisy image $f$, the reaction term can be removed after designing an appropriate $g$ \cite{zhou2014doubly}. 
The solution $u(x, t)$ represents the restored image with the scale variable $t$. 

Our proposed model represents a coupling of adaptive local TV and adaptive NLTV, which corresponds to the limiting case of $p \rightarrow 1^+$ in the following $p$-laplacian ($1<p<\infty$) equation 
\begin{equation}\label{eq:p-lap}
    \left\{\begin{array}{l}
    \frac{\partial u}{\partial t}=\operatorname{div}\left((1-g )|\nabla u|^{p-2} \nabla u\right)+\lambda \int_{\Omega} \frac{g(x)+g(y)}{2} \times \\
    \quad J(x-y)|u(y, t)-u(x, t)|^{p-2}(u(y, t)-u(x, t)) d y, \quad(x, t) \in Q_T, \\
    \frac{\partial u}{\partial\vec{n}}=0, \quad(x, t) \in \partial \Omega \times(0, T), \\
    u(x, 0)=f(x), \quad x \in \Omega.
    \end{array}\right.
    \end{equation}
The $p$-laplacian equation \eqref{eq:p-lap} has been analyzed in \cite{shi2021coupling} for the removal of additive Gaussian noise.

In comparison to the $p$-laplacian model \eqref{eq:p-lap}, the solution of our model belongs to the function space of bounded variation ($BV(\Omega)$), which allows  discontinuities. 
Furthermore, theoretical analysis becomes more challenging compared to \eqref{eq:p-lap} due to the lack of compactness and separability of  $BV(\Omega)$. 
Finally, an appropriately designed function  $g(x)= g(f(x))$, serving as a grayscale indicator as suggested in \cite{dong2013convex}, is employed to guide the diffusion speed of \eqref{eq:tv_case}, thereby enhancing the efficiency of multiplicative noise removal.

The structure of this paper is organized as follows. Section \ref{sec:preli} presents essential preliminaries, along with existence and uniqueness results for our model, as well as several  propositions. 
In Section \ref{sec:pr}, we provide proofs of our results. 
Finally, Section \ref{sec:experi} presents the numerical scheme of our model and comparative experiments against other multiplicative denoising models.

\section{ Preliminaries }\label{sec:preli}

\subsection{Basic assumptions}
Let $\Omega \subset \mathbb{R}^N(N \geqslant 2)$ be a bounded domain with a smooth boundary $\partial \Omega$, $ Q_T=\Omega \times(0, T)$ and $\vec{n}$ denotes the unit outward normal  vector on $\partial \Omega$.  
The kernel function $J: \mathbb{R}^N \rightarrow \mathbb{R}$ in \eqref{eq:tv_case} is a nonnegative, continuous, radially symmetric function with compact support, satisfying $J(0)>0$ and $\int_{\mathbb{R}^N} J(z) dz=1$. 
The weight $g$ in the adaptive total variation satisfies $g \in C(\overline{\Omega})$ and $0 <g_0 \leqslant g(x) \leqslant g_1 <1$. 
Throughout this work, the noisy image $f\in BV(\Omega)\cap L^2(\Omega)$.  

\subsection{The solution space}
The function space of bounded variation $BV(\Omega)$ is a class of functions $u \in L^1(\Omega)$ whose generalized  partial derivatives are measures with finite total variation in $\Omega$. 
Thus 
$u \in BV(\Omega)$ if and only if  $u \in L^1(\Omega)$ and there are Radon measures $\mu_1, \cdots, \mu_N$ defined in $\Omega$ with finite total mass in $\Omega$ and
$$
\int_{\Omega} u D_i \varphi d x=-\int_{\Omega} \varphi d \mu_i, \quad i=1, \cdots, N,
$$
for all $\varphi \in C_0^{\infty}(\Omega)$. Thus the gradient of $u$ is a vector valued measure with finite total variation
\begin{equation}
    |D u|=\sup \left\{\int_{\Omega} u \operatorname{div} \varphi d x: \varphi \in C_0^1\left(\Omega ; \mathbb{R}^N\right),|\varphi(x)| \leqslant 1 \text { for } x \in \Omega\right\} .
    \end{equation}
The space $B V(\Omega)$ is endowed with the norm $\|u\|_{B V(\Omega)}=\|u\|_{L^1(\Omega)}+|D u|$.

Several results from \cite{iindex} are needed in the proof. Let 
$$
X_{p}(\Omega):=\left\{\mathbf{z} \in L^{\infty}\left(\Omega; \mathbb{R}^{N}\right): \operatorname{div}(\mathbf{z}) \in L^{p}(\Omega)\right\}, \quad 1\leqslant p \leqslant N. 
$$
The distribution $(\mathbf{z}, Du) \in \mathcal{D}^\prime(\Omega)$ can be defined as  
\begin{equation*}
    \langle(\mathbf{z}, D u), \varphi\rangle=-\int_{\Omega}  \mathbf{z} \cdot \nabla \varphi ud x -\int_{\Omega}  \operatorname{div}(\mathbf{z})  \varphi u d x,
\end{equation*}
when $\mathbf{z} \in X_p(\Omega)$ and $u \in B V(\Omega) \cap L^{p^{\prime}}(\Omega)$. 
Then $(\mathbf{z}, Du)$ is a Radon measure in $\Omega$ with 
\begin{equation}\label{eq:measure_equality}
\left|\int_B(\mathbf{z}, D u)\right| \leqslant \int_B|(\mathbf{z}, D u)| \leqslant \|\mathbf{z}\|_{\infty} \int_B|D u|,
\end{equation}
for any Borel set $B \subseteq \Omega$. Especially
$$
\int_{\Omega}(\mathbf{z}, D u)=\int_{\Omega} \mathbf{z} \cdot \nabla u d x, \quad \forall u \in W^{1,1}(\Omega) \cap L^{\infty}(\Omega).
$$
Moreover, the following Green's formula can be established:
\begin{lemma}[Green's formula \cite{andreu2011local}]\label{gren}
    Assume $\Omega \subset \mathbb{R}^N$ is a bounded domain with Lipschitz boundary $\partial \Omega$.
    If $\mathbf{z} \in X_p(\Omega)$ and $u \in B V(\Omega) \cap L^{p^{\prime}}(\Omega)$,
    then we have
    \begin{equation*}
        \int_{\Omega}(\mathbf{z}, D u) + \int_{\Omega } u \operatorname{div}(\mathbf{z}) d x=\int_{\partial \Omega}[\mathbf{z}, \nu] u d \mathcal{H}^{N-1},
    \end{equation*}
where $[\mathbf{z}, \nu]$ is the weak trace on $\partial \Omega$ of the normal component of $\mathbf{z} \in X_p(\Omega)$. 
\end{lemma}

To define the differential operator $\operatorname{div}\left((1-g)\frac{Du}{|Du|}\right)$, 
the weighted total variation \cite{chen2003minimization} is used in this paper.

\begin{definition}[\cite{chen2003minimization}]
Suppose $g  \in C(\overline{\Omega})$ with $g>0$ on $\overline{\Omega}$. 
A function $u \in L^1(\Omega)$ has bounded $g$-total variation in $\Omega$, if
$$
\int_{\Omega} g  |D u|=\sup \left\{\int_{\Omega} u \operatorname{div} \varphi d x: \varphi \in C_0^1\left(\Omega ; \mathbb{R}^N\right),|\varphi(x)| \leqslant g(x) \text { for } x \in \Omega\right\} <\infty.
$$
\end{definition}

\begin{remark}[\cite{chen2003minimization}]
$u \in L^1(\Omega)$ having bounded $g$-total variation in $\Omega$ implies that $u \in BV(\Omega)$.
\end{remark}

\begin{lemma}[lower semicontinuity \cite{chen2003minimization}]\label{weak_continuous_u}
Assume that $\left\{u_k\right\}_{k \in \mathbb{N}} \subset B V(\Omega)$ with $u_k \rightarrow u$ in $L^1(\Omega)$. Then, $u \in  BV (\Omega)$, and
$$
\int_{\Omega} g \left|D u\right| \leqslant \liminf _{k \rightarrow \infty} \int_{\Omega} g \left|D u_k\right|,
$$
for any $g  \in C(\overline{\Omega})$ with $g>0$ on $\overline{\Omega}$.
\end{lemma}

Throughout this paper the following multivalued functions will be used:
$$
\operatorname{sgn}(r)= \begin{cases}1 & \text { if } r>0 \\ {[-1,1]} & \text { if } r=0 \\ -1 & \text { if } r<0\end{cases}, \text{ and }
\operatorname{sgn}_0(r)= \begin{cases}1 & \text { if } r>0 \\ 0 & \text { if } r=0 \\ -1 & \text { if } r<0\end{cases}.
$$

\subsection{The results of coupling local and nonlocal $p$-laplacian equation}
The results in \cite{shi2021coupling} illustrate the existence and uniqueness of the solution to problem \eqref{eq:p-lap} for $1<p<+\infty$.

\begin{definition}
A function $u \in C\left([0, T] ; L^2(\Omega)\right) \cap L^p\left((0, T) ; W^{1, p}(\Omega) \right)$ is called a weak solution of \eqref{eq:p-lap} if
\begin{equation}\label{weak_P}
    \begin{aligned}
\int_0^{t_1}\int_{\Omega} \frac{\partial u}{\partial t} \varphi d x dt+&\int_0^{t_1} \int_{\Omega}(1-g)|\nabla u|^{p-2} \nabla u \cdot \nabla \varphi d x d t 
+\frac{\lambda}{2} \int_0^{t_1} \int_{\Omega} \int_{\Omega} g(x) J(x-y) \times
\\
& 
\quad |u(y,t)-u(x,t)|^{p-2} (u(y,t)-u(x,t))(\varphi(y,t)-\varphi(x,t)) d y d x d t=0
    \end{aligned}
\end{equation}
holds for any $t_1\in [0,T]$ and any $\varphi \in C^1\left(\overline{Q_T}\right)$.
\end{definition}

\begin{lemma}[\cite{shi2021coupling}]\label{p-estimate}
Assume that $g \in C(\overline{\Omega})$  with 
    $0 <g_0 \leqslant g(x) \leqslant g_1 <1$. If $f \in L^2(\Omega)$, then \eqref{eq:p-lap} admits a unique solution. Besides, 
\begin{equation}
    \begin{aligned}
    \sup _{0< t< T} \int_{\Omega} u^2 d x &+2 \int_0^T \int_{\Omega}|\nabla u|^p d x d t \\
    &+\lambda \int_0^T \int_{\Omega} \int_{\Omega}  J(x-y)|u(y,t)-u(x,t)|^p d y d x d t \leqslant C \int_{\Omega} |f|^2 d x,  
    \end{aligned}
\end{equation}
where $C$ is independent of $p$.
\end{lemma}

\begin{lemma}\label{p-estimate_2}
Assume that $g \in C(\overline{\Omega})$  with 
    $0 <g_0 \leqslant g(x) \leqslant g_1 <1$. If $f  \in W^{1,p}(\Omega)$, $1<p<+\infty$, u is the solution of \eqref{eq:p-lap}. Then, we have the estimate
\begin{equation}\label{eq:p-estimate_2}
    \begin{aligned}
\int_0^{T}\int_{\Omega} \left(\frac{\partial u}{\partial t} \right)^2  d x dt 
&+\sup _{0< t< T}\int_{\Omega}|\nabla u|^{p} d x\\  
&+\frac{\lambda}{2} \int_{\Omega} \int_{\Omega} J(x-y)|u(y,t)-u(x,t)|^{p}  d y d x \leqslant C\|f\|_{W^{1,p}(\Omega)}.
    \end{aligned}
\end{equation}
\end{lemma}
\begin{proof}
    Taking $\varphi=\frac{\partial u}{\partial t}$ in \eqref{weak_P}, we obtain 
    \begin{equation*}
        \begin{aligned}
    &\int_0^{t_1}\int_{\Omega} \left(\frac{\partial u}{\partial t} \right)^2 d xdt+\int_0^{t_1} \frac{\partial }{\partial t} \left(\int_{\Omega}\frac{1}{p} (1-g)|\nabla u|^{p} -\frac{1}{p} (1-g)|\nabla f|^{p} d x \right) d t \\
    & 
    \quad +\frac{\lambda}{2} \int_0^{t_1} \frac{\partial }{\partial t} \left(\frac{1}{p} \int_{\Omega} \int_{\Omega} g(x) J(x-y)|u(y,t)-u(x,t)|^{p}  d y d x \right)d t=0,  
        \end{aligned}
    \end{equation*}
    since 
    $$
|\nabla u|^{p-2} \nabla u \cdot \nabla u_t =\frac{1}{2} \frac{\partial}{\partial t} \int_0^{|\nabla u(x, t)|^2}s^{(p-2) / 2} ds = \frac{\partial}{\partial t} \left(\frac{1}{p} |\nabla u|^{p} -\frac{1}{p} |\nabla f|^{p}\right).
$$ 
Then, we have 
$$
\begin{aligned}
   \int_0^{t_1}\int_{\Omega} \left(\frac{\partial u}{\partial t} \right)^2 d xdt &+\sup _{0< t< T} \frac{1}{p} \int_{\Omega}(1-g)|\nabla u|^p d x  \\
    &   
    +\frac{\lambda}{2} \frac{1}{p}  \int_{\Omega} \int_{\Omega} g(x) J(x-y)|u(y,t_1)-u(x,t_1)|^{p}  d y d x  \\
    & \quad 
    \leqslant \frac{1}{p} \int_{\Omega}|\nabla f|^p d x + \frac{\lambda}{2} \frac{1}{p}  \int_{\Omega} \int_{\Omega}  J(x-y)|f(y)-f(x)|^{p}  d y d x,  
\end{aligned}
$$
for any $0\leqslant t_1 \leqslant T$. By the non-degeneracy of $g$, \eqref{eq:p-estimate_2} follows. 
\end{proof}

\subsection{The main results}
We first give the definition of the weak solution to problem \eqref{eq:tv_case}.
\begin{definition}\label{def:entro}
    The weak solution to \eqref{eq:tv_case} is a function $u \in C\left([0, T] ; L^2(\Omega)\right) \cap L^{\infty}(0, T ; B V(\Omega)), \frac{\partial u}{\partial t} \in L^2\left(Q_T\right)$, 
    which satisfies $u(x,0)= f(x)$ a.e. on $x \in \Omega$, and
    there exist 
    $\mathbf{z} \in  X_2 \left(Q_T\right)$
    and  $h \in L^{\infty}(\Omega \times \Omega \times ( 0, T ) ) $ with $\|\mathbf{z}\|_\infty \leqslant 1$ and $\|h\|_{\infty} \leqslant 1$ such that $h( x, y,t)=-h( y, x,t)$,   
    \begin{equation}\label{eq:entro_con_2}
    \frac{\partial u}{\partial t}=\operatorname{div}((1-g)\mathbf{z})+\lambda \int_{\Omega} \frac{g(x)+g(y)}{2} J(x-y) h(x, y,t) d y, \text{ in } \mathcal{D}^{\prime}\left(Q_T\right), 
\end{equation}
and
\begin{equation}\label{eq:entro_con_3}
    \begin{aligned}
    \int_{\Omega}(u-\phi) \frac{\partial u}{\partial t} d x 
    \leqslant&
    \frac{\lambda}{2} \int_{\Omega} \int_{\Omega} \frac{g(x)+g(y)}{2}  J(x-y) h(x,y,t)\left(\phi(y,t)-\phi(x,t)\right)d y d x  \\
    & - \frac{\lambda}{2} \int_{\Omega} \int_{\Omega} \frac{g(x)+g(y)}{2}  J(x-y) \left| u(y,t)-u(x,t)\right|   d y d x  \\
    & + \int_{\Omega}((1-g)\mathbf{z} , D \phi)
    -\int_{\Omega}(1-g) |Du|,
    \quad  \text{ a.e. on } [0, T], 
    \end{aligned}
\end{equation}
for every 
$\phi \in L^{\infty}\left(0, T ; BV(\Omega) \cap L^2(\Omega) \right)$.
\end{definition}

\begin{remark}
    The term $\int_{\Omega}((1-g)\mathbf{z} , D \phi) = -\int_{\Omega}\operatorname{div}((1-g)\mathbf{z})  \phi dx$ in \eqref{eq:entro_con_3} by Green's formula and the fact that $[\mathbf{z}, \nu]=0$ holds $\mathcal{H}^{N-1}-\text {a.e. on } \partial \Omega$, which will be verified in Prop. \ref{prop:equivalent}. 
\end{remark}

\begin{theorem}\label{th:exit_unique}
Assume that $g \in C(\overline{\Omega})$  with 
    $0 <g_0 \leqslant g(x) \leqslant g_1 <1$. If $f \in BV(\Omega) \cap L^2(\Omega)$, then problem \eqref{eq:tv_case} admits a unique solution.
\end{theorem}

Some equivalent forms of the weak solution to the problem \eqref{eq:tv_case} can be given to better characterize this solution.
\begin{prop}[Characterize the solution]\label{prop:equivalent}
    Suppose u is the weak solution of  \eqref{eq:tv_case}, then  \eqref{eq:entro_con_3} is equivalent to the following assertions \\
    (a)  
    \begin{gather}
        J(x-y) \frac{g(x)+g(y)}{2} h(x, y,t) \in J(x-y) \frac{g(x)+g(y)}{2} \operatorname{sign}(u(y,t)-u(x,t)), \label{eq:equil_1}\\
        \int_{\Omega}((1-g)\mathbf{z}, D u)= \int_{\Omega}(1-g)|D u|,  \label{eq:equil_1_1}\\
        {[\mathbf{z}, \nu]=0, \quad \mathcal{H}^{N-1}-\text { a.e. on } \partial \Omega. }\label{eq:equil_2} 
    \end{gather}
    (b)  
    \eqref{eq:entro_con_3} holds with equality instead of the inequality. 
\end{prop}
\begin{remark}[\cite{iindex}]
    \eqref{eq:equil_1} is equivalent to
    $$
    \begin{aligned}
    &\frac{1}{2}\int_{0}^{T} \int_{\Omega} \int_{\Omega} \frac{g(x)+g(y)}{2}J(x-y)|u(y,t)-u(x,t)| d y d x dt\\
    &=-\int_{0}^{T}\int_{\Omega} \int_{\Omega} \frac{g(x)+g(y)}{2}J(x-y) h(x, y,t) d y u(x,t) d x dt\\
    &=\frac{1}{2}\int_{0}^{T} \int_{\Omega} \int_{\Omega} \frac{g(x)+g(y)}{2}J(x-y)h(x, y,t)\left(u(y,t)-u(x,t)\right) d y d x dt.
    \end{aligned}
    $$
\end{remark}

\begin{prop}\label{prop:Jg}
    Assume $u_p$ is the weak solution of \eqref{eq:p-lap} for $1<p<\infty$, and $u$ is the weak solution of \eqref{eq:tv_case}. Then 
    \begin{equation}\label{eq:nonlocal_p_2_1}
        \begin{aligned}
        &\lim_{p \rightarrow 1^{+}} \int_0^{T} \int_{\Omega}\int_{\Omega} \frac{g(x)+g(y)}{2} J(x-y)\left|u_{p}(y,t)-u_{p}(x,t)\right|^{p} d y d x dt \\
        &= \int_0^{T} \int_{\Omega} \int_{\Omega} \frac{g(x)+g(y)}{2} J(x-y) |u(y,t)-u(x,t)|dyd x d t .
        \end{aligned}
    \end{equation}
    \end{prop}

\begin{prop}\label{prop:1}
    If the initial value $f \in L^\infty (\Omega)$, then the weak solution of \eqref{eq:tv_case} is bounded with the estimate
    \begin{equation*}
        \|u \|_{L^\infty( \Omega)} \leqslant \|f\|_{L^\infty( \Omega)}, \text{ a.e. } t \in [0,T].
        \end{equation*}
\end{prop}

\begin{prop}\label{prop:2}
    Suppose that $u $ is a weak solution of \eqref{eq:tv_case}, then we have conservation of mass
        $$
        \frac{1}{|\Omega|} \int_{\Omega} u d x=\frac{1}{|\Omega|} \int_{\Omega} f d x,
        $$
    for all $t \geqslant 0 $.
\end{prop}
\begin{proof}
    Taking $\phi = u \pm 1$ in \eqref{eq:entro_con_3}, it follows that $\int_\Omega \frac{\partial u}{\partial t} = 0$. Consequently, the function $t \mapsto \int_{\Omega} u(t)$ is constant. 
\end{proof}

\begin{prop}\label{prop:3}
    Suppose $f \in L^{\infty}(\Omega)$ and   $u$ is the weak solution of \eqref{eq:tv_case}, when $N=2$,  
    there exists a constant $M$ independent of $f$, such that 
    $$
    \left\|u -\overline{f}\right\|_{L^2(\Omega)} \leqslant M \frac{\left\|f\right\|_{L^2(\Omega)}^2}{t} \rightarrow 0, \quad \text { as } t \rightarrow \infty, 
    $$
    where $\overline{f}$ is the mean value of the initial condition $\overline{f}=\frac{1}{|\Omega|} \int_{\Omega} f(x) d x$. 
\end{prop}

Finally, we show that our model can approximate local total variation model under suitable rescaling. More precisely, suppose $u_\varepsilon(x,t)$ is the weak solution to problem  
\begin{equation}\label{eq:TV_nltv_ep}
    \left\{\begin{array}{l}
    \frac{\partial u}{\partial t}=\operatorname{div}\left((1-g)\frac{D u}{|D u|} \right)+\frac{\lambda }{2} \int_{\Omega}  \frac{g(x)+g(y)}{2} J_{\varepsilon}(x-y)\frac{u(y, t)-u(x, t)}{|u(y, t)-u(x, t)|}d y, \quad(x, t) \in Q_T, \\
     \frac{\partial u}{\partial\vec{n}}=0, \quad(x, t) \in \partial \Omega \times(0, T), \\
    u(x, 0)=f(x), \quad x \in \Omega,
    \end{array}\right.
\end{equation}
where the rescaled kernel
$$
J_{\varepsilon}(x)=\frac{C_{J, 1}}{\varepsilon^{1+N}} J\left(\frac{x}{\varepsilon}\right), \quad C_{J, 1}^{-1}=\frac{1}{2} \int_{\mathbb{R}^N} J(z)\left|z_1\right| d z.
$$
Moreover, we suppose $u(x,t)$ is the weak solution to  the total variation model 
\begin{equation}\label{eq:TV_local}
    \left\{\begin{array}{l}
    \frac{\partial u}{\partial t}=\operatorname{div}\left(\frac{D u}{|D u|} \right), \quad(x, t) \in Q_T, \\
     \frac{\partial u}{\partial\vec{n}}=0, \quad(x, t) \in \partial \Omega \times(0, T), \\
    u(x, 0)=f(x), \quad x \in \Omega.
    \end{array}\right.
\end{equation}
Our
main concern is to show that $u_\varepsilon$ converge to $u$ as $\varepsilon \rightarrow 0$.
\begin{theorem}\label{th:nl2local}
    Let $f \in L^2(\Omega) $. Assume that $J$ satisfies $J(x) \geqslant J(y)$ for $|x| \leqslant|y|$, $g \in C(\overline{\Omega})$  with 
    $0 <g_0 \leqslant g(x) \leqslant g_1 <1$. Then
    $$
    \lim _{\varepsilon \rightarrow 0} \sup _{t \in[0, T]}\left\|u_{\varepsilon}(x, t)-u(x, t)\right\|_{L^1(\Omega)}=0.
    $$
\end{theorem}

\section{Proof of main results}\label{sec:pr}

\subsection{Proof of Theorem \ref{th:exit_unique}}
For initial value $f \in BV(\Omega)\cap L^2(\Omega)$, there exists a sequence $\{f_p\} \subset  C^{\infty}(\Omega) \cap B V(\Omega) \cap L^2(\Omega)$ \cite{evans2018measure}, such that  
\begin{equation}\label{eq:initial_conve}
f_p \rightarrow f, \text{ in } L^2(\Omega), \text{ and }
\int_{\Omega} |\nabla f_p|dx \rightarrow \int_{\Omega} |Df|.
\end{equation}
Suppose $u_p$ is the weak solution to problem \eqref{eq:p-lap} with initial value $u_p(x,0)=f_p(x)$, and $1<p<2$. Combining Lemma \ref{p-estimate}, Lemma \ref{p-estimate_2}  and $0 <1-g_1 \leqslant 1- g(x) \leqslant 1-g_0 <1$, we deduce 
\begin{equation}\label{eq:estimate_1}
    \begin{aligned}
&\int_0^T \int_{\Omega} \left( \frac{\partial u_p}{\partial t} \right)^2 dxdt+ \int_0^T \int_{\Omega} \int_{\Omega}  J(x-y)|u_p(y,t)-u_p(x,t)|^p d y d x d t \\
&\quad \quad \quad \quad \quad \quad \quad \quad \quad \  + \sup _{0<t<T } \int_{\Omega} |u_p(x, t)|^2 d x +  \sup _{0<t<T } \int_{\Omega}|\nabla u_p|^p d x   \\
&\leqslant C\left(\|f_p\|_{W^{1,p}(\Omega)}+\|f_p\|_{L^{2}(\Omega)}\right) \leqslant C\left(\|f\|_{BV(\Omega)}+\|f\|_{L^{2}(\Omega)}\right) \leqslant M,
    \end{aligned}
\end{equation}
where $M$ is independent of $p$.
Then, there exists a function $u \in L^{\infty}(0, T ; B V(\Omega)) \cap C\left([0, T] ; L^2(\Omega)\right)$ such that 
\begin{gather}
     \int_{\Omega} (1-g)|D u| \leqslant \liminf_{p \rightarrow 1^{+}} \int_{\Omega} (1-g)|\nabla u_p|dx,\label{eq:semi_1}\\
    \frac{\partial u_p}{\partial t} \rightharpoonup  \frac{\partial u}{\partial t} ,\quad \text { in } L^2\left(Q_T\right).\label{eq:l2_weak}
\end{gather}

Since $\{u_p\}$ is bounded in $L^{\infty}(0, T ; B V(\Omega))$, $\{\frac{\partial u_p}{\partial t}\}$ is bounded in $L^{2}(0, T ; L^1(\Omega))$, and the compact embedding of $BV(\Omega)$ into $L^s(\Omega)$ for any $s < \frac{N}{N-1}$, it follows that $\{u_p\}$ is compact in $C(0,T; L^s(\Omega))$ according to the compact result (corollary 4, \cite{simon1986compact}). 
Then, by the compactly embedding of $C(0,T; L^s(\Omega))$ into $L^s(Q_T)$, we have 
\begin{equation}\label{eq:L1_strong}
    u_p \rightarrow u, \quad \text { in } L^1\left(Q_T\right).
\end{equation}

Now, we prove \eqref{eq:entro_con_2}. For any  measurable subset $E \subset Q_T$, employing  h{\"o}lder's inequality and \eqref{eq:estimate_1}, we deduce 
\begin{equation*}
    \iint_{E}|\nabla u_p|^{p-1}  d x dt
    \leqslant  \left(\int_0^T \int_{\Omega} \left|\nabla u_p\right|^p d x dt \right)^{(p-1) / p}|E|^{1 / p}
    \leqslant  M^{(p-1) / p}\left|E\right|^{1 / p},
\end{equation*}

which implies that  $\{\left|\nabla u_p\right|^{p-2} \nabla u_p \}$  is equi-integrable in $L^1\left(Q_T; \mathbb{R}^N\right)$. According to the Dunford-Pettis Theorem, there exists $\mathbf{z} \in L^1\left(Q_T; \mathbb{R}^N\right)$ such that
\begin{equation}\label{eq:local_p}
\left|\nabla u_p\right|^{p-2} \nabla u_p \rightharpoonup \mathbf{z}, \quad \text { in } L^1\left(Q_T; \mathbb{R}^N\right). 
\end{equation}
Similarly for any measurable subset $G \subset \Omega \times \Omega \times (0,T)$, we have 
\begin{equation*}
    \begin{aligned}
    &\left|  \iiint_{G}   J(x-y) \left| u_{p}(y,t)-u_{p}(x,t)\right|^{p-2}\left(u_{p}(y,t)-u_{p}(x,t) \right)dydxdt\right|\\
    &\leqslant   \iiint_{G} J(x-y)\left|u_{p}(y,t)-u_{p}(x,t) \right|^{p-1}dydxdt \\
    &\leqslant M^{(p-1) / p}|G|^{1 / p},
    \end{aligned}
\end{equation*}

Then there exists $h(x, y,t) \in L^1\left(\Omega \times \Omega \times (0,T)\right)$ with $h(y,x,t)=-h(x, y,t)$ such that 
\begin{equation}\label{nonlocal_p}
    J(x-y)\left| u_{p}(y,t) - u_{p}(x,t)\right|^{p-2}\left(u_{p}(y,t) - u_{p}(x,t)\right) \rightharpoonup J(x-y) h(x, y,t)
\end{equation}
holds in $L^1\left(\Omega \times \Omega \times (0,T)\right)$.
Combining now \eqref{eq:l2_weak}, \eqref{eq:local_p}, \eqref{nonlocal_p}   and letting $p \rightarrow 1^+$ in \eqref{weak_P} we find  
\begin{equation*} 
    \begin{aligned}
    \int_0^{T}\int_{\Omega} \frac{\partial u}{\partial t} \varphi d x&+\int_0^{T} \int_{\Omega}(1-g)\mathbf{z} \cdot \nabla \varphi d x d t \\
    &+\frac{\lambda}{2} \int_0^{T} \int_{\Omega} \int_{\Omega} g(x) J(x-y) h(x, y,t) \left( \varphi(y,t) -\varphi(x,t) \right) d y d x d t=0, 
    \end{aligned}
    \end{equation*}
holds for any $\varphi \in C^1\left(\overline{Q_T}\right)$. 
Therefore \eqref{eq:entro_con_2} follows from  the equality 
\begin{equation}\label{eq:nonlocal_int_by_parts}
    \begin{aligned}
    &\frac{1}{2} \int_0^{T} \int_{\Omega} \int_{\Omega} g(x) J(x-y) h(x, y,t) \left( \varphi(y,t) -\varphi(x,t) \right) d y d x d t\\
    &=\frac{1}{2} \int_0^{T} \int_{\Omega} \int_{\Omega} \frac{g(x)+g(y)}{2} J(x-y) h(x, y,t) \left( \varphi(y,t) -\varphi(x,t) \right) d y d x d t\\
    &=-\int_0^{T} \int_{\Omega} \int_{\Omega} \frac{g(x)+g(y)}{2} J(x-y) h(x, y,t) dy  \varphi(x,t)dxdt.
    \end{aligned}
\end{equation}
For simplicity, the \eqref{eq:nonlocal_int_by_parts} is called the \emph{nonlocal integrate by parts formula} in our paper.

Next we will show $\|\mathbf{z}\|_\infty \leqslant 1$. Let 
$B_{p, m}=\left\{(x, t) \in Q_T:\left|\nabla u_p\right|>m\right\}$ for any $m>0$ and $1<p<2$. Then we have
\begin{equation*}
    m^p |B_{p, m}| \leqslant \iint_{B_{p, m}} |\nabla u_p|^{p} dxdt \leqslant M.
\end{equation*}
As above, there exists a function $r_m \in L^1\left(Q_T; \mathbb{R}^N\right)$ such that
$$
\left|\nabla u_p\right|^{p-2} \nabla u_p \chi_{B_{p, m}} \rightharpoonup  r_m, \quad \text { in } L^1\left(Q_T; \mathbb{R}^N\right), 
$$
where $\chi_{B_{p, m}}$ is the indicator of $B_{p, m}$. Now for any $\varphi \in L^{\infty}\left(Q_T; \mathbb{R}^N\right)$ with $\|\varphi\|_{L^{\infty}\left(Q_T; \mathbb{R}^N\right)} \leqslant 1$, 
\begin{equation*}
    \left|\int_0^T \int_{\Omega} |\nabla u_p|^{p-2} \nabla u_p \cdot \varphi \chi_{B_{p, m}}  dxdt \right| \leqslant  \iint_{B_{p, m}} |\nabla u_p|^{p-1} dxdt \leqslant \frac{M}{m}.
\end{equation*}
Letting $p \rightarrow 1^{+}$, we deduce 
\begin{equation}\label{eq:rm}
\int_0^T \int_{\Omega}\left|r_m\right| d x d t \leqslant \frac{M}{ m}, \quad \text { for any } m>0.
\end{equation}
On the other hand, for any measurable set $E \subset Q_T$ and any $\varphi \in L^{\infty}\left(Q_T; \mathbb{R}^N\right)$ with $\|\varphi\|_{L^{\infty}\left(Q_T; \mathbb{R}^N\right)} \leqslant 1$, we have 
\begin{equation}\label{eq:sm}
\left|\iint_{E} |\nabla u_p|^{p-2} \nabla u_p \cdot \varphi \chi_{Q_T / B_{p, m}} dxdt \right| 
\leqslant \iint_{E \cap  Q_T / B_{p, m}} |\nabla u_p|^{p-1}dxdt 
 \leqslant m^{p-1} |E|. 
\end{equation}
Letting $p \rightarrow 1^{+}$, we obtain that $\left|\nabla u_p\right|^{p-2} \nabla u_p  \chi_{Q_T / B_{p, m}}$ weakly converges to some function $s_m \in L^1\left(Q_T; \mathbb{R}^N\right)$. Let  $p \rightarrow 1^{+}$ in \eqref{eq:sm}, to deduce
$\left|\iint_{E}s_m \cdot \varphi dxdt \right|\leqslant  |E|$. 
Taking  $\varphi=\frac{s_m}{|s_m|}$, we have 
$$
\iint_{E}|s_m| dxdt \leqslant  |E| , 
$$
which implies $\left\|s_m\right\|_{\infty}  \leqslant 1$ due to the arbitrariness of $E$. For any $m>0$, $\mathbf{z}$  can be decomposed into $r_m+s_m$ with $\left\|s_m\right\|_{\infty} \leqslant 1$ and $r_m$ satisfies \eqref{eq:rm}. Then  
\begin{equation*}
    \left|\iint_{E}\mathbf{z} \cdot \varphi dxdt \right|\leqslant  
    \left|\iint_{E}s_m \cdot \varphi dxdt \right| + \left|\iint_{E}r_m \cdot \varphi dxdt \right|\leqslant  |E| + \frac{M}{m}, \quad \text { for any } m>0. 
\end{equation*}
Letting $m\rightarrow \infty$, we conclude from the arbitrariness of $E$ that  
$\|\mathbf{z}\|_{\infty}  \leqslant 1$.

Similarly we have $\|h\|_\infty \leqslant 1$.
Let $B_{p, m}=\{(x,y, t) \in \Omega \times \Omega \times (0,T):\left|u_p(y,t)-u_p(x,t)\right|>m\}$ for any $m>0$ and $1<p<2$. Then 
\begin{equation*}
     m^p \iiint_{B_{p, m}}  J(x-y) dydxdt \leqslant \iiint_{B_{p, m}}J(x-y) |u_p(y,t)-u_p(x,t)|^{p} dydxdt \leqslant M.
\end{equation*}
Similar to the proof of $\|\mathbf{z}\|_\infty \leqslant 1$, there exists functions $r_m$, $s_m \in L^1\left(\Omega \times \Omega \times (0,T)\right)$ such that
\begin{gather*}
J(x-y)\left|u_p(y,t)-u_p(x,t) \right|^{p-2}\left(u_p(y,t)-u_p(x,t)\right) \chi_{B_{p, m}} \rightharpoonup  J(x-y) r_m\\
J(x-y)\left|u_p(y,t)-u_p(x,t)\right|^{p-2}(u_p(y,t)-u_p(x,t))  \chi_{Q_T /B_{p, m}} \rightharpoonup  J(x-y)s_m
\end{gather*}
hold in $L^1\left(\Omega \times \Omega \times (0,T)\right)$, 
where $\chi_{B_{p, m}}$ is the indicator of $B_{p, m}$. Besides, the $r_m$ and $s_m$ satisfy
\begin{equation}\label{eq:rm_nonlocal}
    \int_0^T \int_{\Omega}\int_{\Omega}\left|J(x-y)r_m\right| dy d x d t \leqslant \frac{M}{ m}, \quad \text { for any } m>0,  
\end{equation}
and 
\begin{equation}\label{eq:sm_inequal}
\iiint_{G} J(x-y)\left|s_m\right| dydxdt 
    \leqslant  \iiint_{G}J(x-y) dydxdt, 
\end{equation}
for any measurable set $G \subset \Omega \times \Omega \times (0,T)$. Denote 
$$
U := \{(x,y): \ x-y\in \text{supp}J, \ x \in \Omega, \ y \in \Omega\}.   
$$
$\left\|s_m\right\|_{L^\infty(U \times (0,T))} \leqslant 1$ follows from 
\eqref{eq:sm_inequal} and the arbitrariness of $G$.  
For any $m>0$, $h$  can be decomposed into $r_m+s_m$ with $\left\|s_m\right\|_{\infty} \leqslant 1$ and $r_m$ satisfies \eqref{eq:rm_nonlocal}. Then  
\begin{equation*}
    \begin{aligned}
    \left|\iiint_{G}J(x-y) h   dydxdt \right|&\leqslant  
    \left|\iiint_{G}J(x-y)s_m  dydxdt \right| + \left|\iiint_{G}J(x-y)r_m   dydxdt \right|\\
    &\leqslant  \iiint_{G} J(x-y) dydxdt + \frac{M}{m}. 
    \end{aligned}
\end{equation*}
Letting $m\rightarrow \infty$, we conclude from the arbitrariness of $G$ that $\left\|h\right\|_{L^\infty(U \times (0,T))} \leqslant 1$. 
After setting $h=0$ out of $U$, $\left\|h\right\|_\infty \leqslant 1$ follows.

Before illustrating \eqref{eq:entro_con_3}, some convergences need to be considered. 
By the convergence of initial values in \eqref{eq:initial_conve}, we have 
$ \lim_{p \rightarrow 1^{+}} \left.\int_{\Omega} (u-u_p)^2 dx \right|_0^T \geqslant 0
$. 
Combining \eqref{eq:l2_weak} and \eqref{eq:L1_strong}, we discover 
\begin{equation}\label{eq:partial_u_p_t}
    \begin{aligned}
    \lim _{p \rightarrow 1^+} -\int_0^{T}\int_{\Omega} \frac{\partial u_p}{\partial t} u_p d x dt&= \int_0^{T}\int_{\Omega} \frac{\partial u}{\partial t} (u-u_p) dx dt - \int_0^{T}\int_{\Omega} \frac{\partial u_p}{\partial t} udxdt \\
    &\quad - \frac{1}{2} \left.\int_{\Omega} (u-u_p)^2 dx \right|_0^T
    \leqslant - \int_0^{T}\int_{\Omega} \frac{\partial u}{\partial t} udxdt. 
    \end{aligned}
\end{equation}
Besides, 
appling weighted h{\"o}lder's inequality, Fatou's Lemma, and having in mind \eqref{eq:L1_strong},   we have
\begin{equation}\label{other_hand_result}
    \begin{aligned}
    &\int_0^{T} \int_{\Omega} \int_{\Omega} \frac{g(x)+g(y)}{2} J(x-y) |u(y,t)-u(x,t)|dyd x d t \\
    &\leqslant
    \varliminf_{p \rightarrow 1^{+}}
    \int_0^{T} \int_{\Omega}\int_{\Omega} \frac{g(x)+g(y)}{2} J(x-y)\left|u_{p}(y,t)-u_{p}(x,t)\right|^{p} d y d x dt.  
    \end{aligned}
\end{equation}

According to the density of $C^{\infty}(\overline{Q_T}) $ in $C\left([0, T] ; L^2(\Omega)\right) \cap L^p\left((0, T) ; W^{1, p}(\Omega) \right)$, 
for any $\varphi \in C^{\infty}(\overline{Q_T})$ with $\varphi(x, 0)=\varphi(x, T)=0$, if we take  $\varphi=\left(u_p-\phi\right) \eta(t)$ with $\eta(t) \geqslant 0$ as the test function in \eqref{weak_P}, it follows that 
\begin{equation*}
    \begin{aligned}
    &\int_0^{T}\int_{\Omega} \frac{\partial u_p}{\partial t} \left(u_p-\phi\right) \eta(t) d x\\
    &=-\int_0^{T} \int_{\Omega} (1-g)|\nabla u_p|^{p-2} \nabla u_p \cdot \nabla \left(u_p-\phi \right) \eta(t) d x d t \\
    &\quad -\frac{\lambda}{2} \int_0^{T} \int_{\Omega} \int_{\Omega} \frac{g(x)+g(y)}{2} J(x-y)| u_p(y,t)-u_p(x,t)|^{p} \eta(t) d y d x d t\\
    &\quad +\frac{\lambda}{2} \int_0^{T} \int_{\Omega} \int_{\Omega} \frac{g(x)+g(y)}{2} J(x-y)|u_p(y,t)-u_p(x,t)|^{p-2} \times \\
    &\quad \quad \quad \quad \quad \left(u_p(y,t)-u_p(x,t) \right) \left(\phi (y,t)-\phi (x,t)\right) \eta(t) d y d x d t. 
    \end{aligned}
\end{equation*}
Take the lower limit for $p \rightarrow 1^+$, we have 
\begin{equation}\label{eq:entropy_re}
    \begin{aligned}
    &\int_0^{T}\int_{\Omega} \frac{\partial u}{\partial t} \left(u-\phi \right) \eta(t) d x\\
    &\leqslant
    \int_0^{T} \int_{\Omega} (1-g)\mathbf{z} \cdot \nabla \phi  \eta(t) dxdt 
    -\int_0^{T} \int_{\Omega}(1-g) |Du| \eta(t)  d t \\
    &\quad -\frac{\lambda}{2} \int_0^{T} \int_{\Omega} \int_{\Omega} \frac{g(x)+g(y)}{2} J(x-y) \left|u(y,t)-u(x,t)\right|\eta(t) d y d x d t\\
    &\quad +\frac{\lambda}{2} \int_0^{T} \int_{\Omega} \int_{\Omega} \frac{g(x)+g(y)}{2} J(x-y)h(x,y,t) 
    \left(\phi (y,t)-\phi (x,t) \right) \eta(t) d y d x d t.
    \end{aligned}
\end{equation} 
Due to the density of $C^{\infty}(\overline{Q_T}) $ in $L^2(0,T; W^{1,1}(\Omega) \cap L^2(\Omega))$, 
\eqref{eq:entropy_re} valid for any $\phi \in L^2(0,T; W^{1,1}(\Omega) \cap L^2(\Omega))$. 

Now, let $\phi \in L^\infty(0,T;  B V(\Omega) \cap L^2(\Omega))$, then $\phi \in L^2(0,T;  B V(\Omega) \cap L^2(\Omega))$, and according to the Theorem B.3 and Lemma C.8 in \cite{iindex}, there exists a sequence $\{\phi_n\} \subset L^2(0,T; B V(\Omega) \cap L^2(\Omega)) \cap C^{\infty}(Q_T) $ such that
\begin{gather*}
    \phi_n \rightarrow \phi,  \text { in } L^2(0,T;L^2(\Omega)), \quad
    \phi_n \rightarrow \phi,  \text { in } L^2(0,T;L^2(\partial \Omega, \mathcal{H}^{N-1})),\\
    \int_{0}^{T}\int_{\Omega}(1-g)\left|\nabla \phi_n\right| d xdt \rightarrow \int_{0}^{T}\int_{\Omega}(1-g)|D \phi| dt.
    \end{gather*}
In \eqref{eq:entro_con_2}, due to the continuity of $g$ and $J$, and the boundedness of $h$, we conclude that $\operatorname{div}((1-g)\mathbf{z} ) \in L^2 (Q_T)$. 
Then  using Green's formula, it follows that 
\begin{equation*}
    \int_{0}^{T}\int_{\Omega} (1-g)\mathbf{z} \cdot \nabla \phi_n \eta(t) d x dt
    \rightarrow \int_{0}^{T}\int_{\Omega}((1-g)\mathbf{z} , D \phi) \eta(t) dt. 
\end{equation*}
We use $\phi_n$ as test function in  \eqref{eq:entropy_re} and let $n \rightarrow \infty$ to obtain 
\begin{equation*}
    \begin{aligned}
    &\int_0^{T}\int_{\Omega} \frac{\partial u}{\partial t} \left(u-\phi \right) \eta(t) d x\\
    &\leqslant
    \int_{0}^{T}\int_{\Omega}((1-g)\mathbf{z} , D \phi) \eta(t) dt 
    -\int_0^{T} \int_{\Omega}(1-g) |Du| \eta(t)  d t \\
    &\quad -\frac{\lambda}{2} \int_0^{T} \int_{\Omega} \int_{\Omega} \frac{g(x)+g(y)}{2} J(x-y) \left|u(y,t)-u(x,t)\right|\eta(t) d y d x d t\\
    &\quad +\frac{\lambda}{2} \int_0^{T} \int_{\Omega} \int_{\Omega} \frac{g(x)+g(y)}{2} J(x-y)h(x,y,t) 
    \left(\phi (y,t)-\phi (x,t) \right) \eta(t) d y d x d t, 
    \end{aligned}
\end{equation*} 
for any $\phi \in L^\infty(0,T;  B V(\Omega) \cap L^2(\Omega))$. 
From the arbitrariness of $\eta(t)$, \eqref{eq:entro_con_3} follows. 

Finally, we prove the uniqueness of the weak solution. Let $u_1$, $u_2$ are weak solutions to \eqref{eq:tv_case} with initial values $f_{1}$, $f_{2}$. Then there exist $h_1$, $h_2 \in L^{\infty}\left( \Omega \times \Omega \times (0,T)\right)$ and $\mathbf{z}_1$,  $\mathbf{z}_2 \in X_2 \left(Q_T\right)$ such that
\begin{equation}\label{eq:unique_1}
\begin{aligned}
    \int_{\Omega}(u_1-\phi) \frac{\partial u_1}{\partial t} dx
    \leqslant 
    &\frac{\lambda}{2} \int_{\Omega} \int_{\Omega} \frac{g(x)+g(y)}{2} J(x-y) h_1(x,y,t)\left(\phi(y,t)-\phi(x,t) \right)  d y d x \\
    &-\frac{\lambda}{2} \int_{\Omega} \int_{\Omega} \frac{g(x)+g(y)}{2} J(x-y) \left|u_1(y,t)-u_1(x,t)\right| d y d x \\
    &+\int_{\Omega} \left((1-g)\mathbf{z}_1, D \phi\right)  
    -\int_{\Omega}(1-g) |Du_1|, \quad  \text{ a.e. on } [0, T] ,
    \end{aligned}
\end{equation}
and 
\begin{equation}\label{eq:unique_2}
\begin{aligned}
    \int_{\Omega}(u_2-\phi) \frac{\partial u_2}{\partial t} dx
    \leqslant 
    &\frac{\lambda}{2} \int_{\Omega} \int_{\Omega} \frac{g(x)+g(y)}{2} J(x-y) h_2(x,y,t)\left(\phi(y,t)-\phi(x,t) \right)  d y d x \\
    &-\frac{\lambda}{2} \int_{\Omega} \int_{\Omega} \frac{g(x)+g(y)}{2} J(x-y) \left|u_2(y,t)-u_2(x,t)\right| d y d x \\
    &+\int_{\Omega} \left((1-g)\mathbf{z}_2, D\phi \right) 
    -\int_{\Omega}(1-g) |Du_2|, \quad  \text{ a.e. on } [0, T] ,
    \end{aligned}
\end{equation}
for any $\phi \in L^\infty(0,T;  B V(\Omega) \cap L^2(\Omega))$. Taking test function  $\phi=u_{2}$ in \eqref{eq:unique_1}, $\phi=u_{1}$ in \eqref{eq:unique_2} and adding these two inequalities, we obtain 
\begin{equation*}
    \begin{aligned}
    &\int_{\Omega}\left(u_1-u_2\right)\left(\frac{\partial u_1}{\partial t}-\frac{\partial u_2}{\partial t}\right) dx \\
    &\leqslant  \int_{\Omega} \left((1-g)\mathbf{z}_1, D u_{2} \right)
    -\int_{\Omega}(1-g) |Du_2|
    + \int_{\Omega} \left((1-g)\mathbf{z}_2, D u_{1} \right)
    -\int_{\Omega}(1-g) |Du_1|\\
    &\quad +\frac{\lambda}{2} \int_{\Omega} \int_{\Omega} \frac{g(x)+g(y)}{2} J(x-y) \times\\
    &\quad\quad\quad\quad\quad\quad\quad\quad\quad\quad\left(h_1(x,y,t)(u_{2}(y,t) - u_{2}(x,t))-\left|u_2(y,t) - u_2(x,t) \right|   \right)d y d x \\
    &\quad +\frac{\lambda}{2} \int_{\Omega} \int_{\Omega} \frac{g(x)+g(y)}{2}J(x-y) \times \\
    &\quad\quad\quad\quad\quad\quad \quad\quad\quad\quad\left(h_2(x,y,t)(u_{1}(y,t) - u_{1}(x,t))-\left| u_1(y,t) - u_1(x,t) \right|   \right)d y d x \\
    &\leqslant 0,
    \end{aligned}
\end{equation*}
according to $\|\mathbf{z}\|_\infty \leqslant 1$ and $\|h\|_\infty \leqslant 1$. 
Integrating it from 0 to $t$ and letting $n \rightarrow \infty$, it follows that 
\begin{equation*}
    \int_{\Omega}\left(u_1-u_2\right)^2 d x \leqslant \int_{\Omega}\left(f_{1}-f_{2}\right)^2 dx.
\end{equation*}

\subsection{Proof of propositions}

\begin{proof}[\textbf{Proof of Proposition \ref{prop:equivalent}}]

    To prove (b) from (a), multiply $(u-\phi)\eta(t)$ to the both sides of \eqref{eq:entro_con_2} with $\phi \in L^\infty(0,T;  B V(\Omega) \cap L^2(\Omega))$, $\eta(t) \in C[0,T]$ and $\eta(t) \geqslant0$. Then using Green's formula is enough. In addition, it is clearly that (b) implies \eqref{eq:entro_con_3}.

    The only thing we need to prove is that (a) follows from \eqref{eq:entro_con_3}. Choosing $\phi=u$ in \eqref{eq:entro_con_3}, we have 
    \begin{equation*}
        \begin{aligned}
            &\frac{\lambda}{2} \int_{\Omega} \int_{\Omega} \frac{g(x)+g(y)}{2}  J(x-y) \left| u(y,t)-u(x,t)\right|   d y d x + \int_{\Omega}(1-g)|D u| \\
            &\leqslant \frac{\lambda}{2} \int_{\Omega} \int_{\Omega} \frac{g(x)+g(y)}{2}  J(x-y) h(x,y,t)\left(u(y,t)-u(x,t)\right)  d y d x + \int_{\Omega}((1-g)\mathbf{z}, D u).
            \end{aligned}
    \end{equation*}
    Combining the fact that  $\|\mathbf{z}\|_\infty  \leqslant 1$ and  $\|h\|_\infty  \leqslant 1$, we have 
    \begin{equation*}
        \begin{aligned}
            &\int_{\Omega} \int_{\Omega} \frac{g(x)+g(y)}{2}  J(x-y) h(x,y,t)\left(u(y,t)-u(x,t)\right)  d y d x \\
            &\quad \leqslant \|h\|_\infty  \int_{\Omega} \int_{\Omega} \frac{g(x)+g(y)}{2}  J(x-y) \left| u(y,t)-u(x,t)\right|   d y d x \\
            &\quad \leqslant \int_{\Omega} \int_{\Omega} \frac{g(x)+g(y)}{2}  J(x-y) \left| u(y,t)-u(x,t)\right|   d y d x, 
            \end{aligned}
    \end{equation*}
    and 
    $$
    \int_{\Omega}((1-g)\mathbf{z}, D u) \leqslant \|\mathbf{z}\|_\infty \int_{\Omega}(1-g)|D u| \leqslant \int_{\Omega}(1-g)|D u|,
    $$
    which imply \eqref{eq:equil_1} and \eqref{eq:equil_1_1}.
    Finally, we select $\phi =u \pm \varphi$ for $\varphi \in C^{\infty}(\overline{Q_T})$ and $\varphi \geqslant 0$ in \eqref{eq:entro_con_3}, and use Green's formula to obtain 
    \begin{equation*} 
        \begin{aligned}
            \mp  \int_{\Omega} \frac{\partial u}{\partial t}  \varphi d x 
            &\leqslant 
            \pm \frac{\lambda}{2} \int_{\Omega} \int_{\Omega} \frac{g(x)+g(y)}{2}  J(x-y) h(x,y,t)\left(\varphi(y,t)-\varphi(x,t)\right)d y d x  \\
            &\quad \mp \int_{\Omega} \operatorname{div}\left( (1-g)\mathbf{z}\right) \varphi dx 
            \pm \int_{\partial \Omega} (1-g)[\mathbf{z},\nu] \varphi d \mathcal{H}^{N-1}, 
            \quad  \text{ a.e. on } [0, T] ,
            \end{aligned}    
    \end{equation*}
    which implies \eqref{eq:equil_2}.
    \end{proof}

\begin{proof}[\textbf{Proof of Proposition \ref{prop:Jg}}]
    Firstly, the following formula has been illustrated in Theorem \ref{th:exit_unique}.
    \begin{equation*}
        \begin{aligned}
        &\int_0^{T} \int_{\Omega} \int_{\Omega} \frac{g(x)+g(y)}{2} J(x-y) |u(y,t)-u(x,t)|dyd x d t \\
        &\leqslant
        \varliminf_{p \rightarrow 1^{+}}
        \int_0^{T} \int_{\Omega}\int_{\Omega} \frac{g(x)+g(y)}{2} J(x-y)\left|u_{p}(y,t)-u_{p}(x,t)\right|^{p} d y d x dt .
        \end{aligned}
    \end{equation*}
    In this proof, we illustrate 
    \begin{equation}\label{eq:suplim_u_p}
        \begin{aligned}
        &\varlimsup_{p \rightarrow 1^{+}}
        \int_0^{T} \int_{\Omega}\int_{\Omega} \frac{g(x)+g(y)}{2} J(x-y)\left|u_{p}(y,t)-u_{p}(x,t)\right|^{p} d y d x dt \\
        &\leqslant \int_0^{T} \int_{\Omega} \int_{\Omega} \frac{g(x)+g(y)}{2} J(x-y) |u(y,t)-u(x,t)|dyd x d t .
        \end{aligned}
    \end{equation}

    We first give some estimates of the local part. 
    Observe $(1-g)\mathbf{z} \in X_2(\Omega)$ and $u \in BV(\Omega)\cap  L^2(\Omega)$, using Green's formula and the equivalent form (a) of the weak solution to \eqref{eq:tv_case}, that 
    \begin{equation*}
    - \int_{\Omega}u \operatorname{div}((1-g)\mathbf{z})  d x = \int_{\Omega} ((1-g)\mathbf{z},Du) 
    =\int_{\Omega} (1-g)|D u|.
    \end{equation*}
    Then employing H{\"o}lder's inequality and the lower semicontinuity in Lemma \ref{weak_continuous_u}, we have 
    \begin{equation}\label{eq:green_f}
        \begin{aligned}
        \varlimsup_{p \rightarrow 1^{+}} \int_0^{T} \int_{\Omega} (1-g)|\nabla u_p|^p dxdt 
        &\geqslant \varliminf_{p \rightarrow 1^{+}} \int_0^{T} \int_{\Omega} (1-g) |\nabla u_p| dxdt \\
        &\geqslant \int_0^{T} \int_{\Omega} (1-g)|Du|dt 
        = -\int_0^{T}   \int_{\Omega}u \operatorname{div}((1-g)\mathbf{z})  d xdt.
        \end{aligned}
    \end{equation}

    Now, let us choose $\varphi = u_p$ in \eqref{weak_P}, using the nonlocal integrate by parts formula, taking upper limit, and combining  \eqref{eq:partial_u_p_t} and \eqref{eq:green_f}, to obtain 
    \begin{equation*}
        \begin{aligned}
        &\varlimsup_{p \rightarrow 1^{+}} \frac{\lambda}{2} \int_0^{T} \int_{\Omega}\int_{\Omega} \frac{g(x)+g(y)}{2} J(x-y)\left|u_{p}(y,t)-u_{p}(x,t)\right|^{p} d y d x dt \\
        &\leqslant- \int_0^{T}\int_{\Omega} \frac{\partial u}{\partial t} udxdt 
        +\int_0^{T} \int_{\Omega}u\operatorname{div}((1-g)\mathbf{z})  d x d t. 
        \end{aligned}
    \end{equation*}
    At the same time, multipling $u$ on the both sides of  \eqref{eq:entro_con_2}, we have 
    \begin{equation*}
        \begin{aligned}
        &\int_0^{T}\int_{\Omega} \frac{\partial u}{\partial t} u d x dt -\int_0^{T} \int_{\Omega}u \operatorname{div}((1-g)\mathbf{z}) d x d t \\
        & =\lambda \int_0^{T} \int_{\Omega} \int_{\Omega} \frac{g(x)+g(y)}{2}  J(x-y) h(x, y,t) d y u(x,t)  d x d t. 
        \end{aligned}
    \end{equation*}
Then 
    \begin{equation*}
        \begin{aligned}
        &\varlimsup_{p \rightarrow 1^{+}} \frac{1}{2} \int_0^{T} \int_{\Omega}\int_{\Omega} \frac{g(x)+g(y)}{2} J(x-y)\left|u_{p}(y,t)-u_{p}(x,t)\right|^{p} d y d x dt \\
        &\leqslant-\int_0^{T} \int_{\Omega} \int_{\Omega}\frac{g(x)+g(y)}{2} J(x-y) h(x, y,t) d y u(x,t) d x dt,
        \end{aligned}
    \end{equation*}    
which coincide with \eqref{eq:suplim_u_p} because of \eqref{eq:equil_1}. 
    \end{proof}

\begin{proof}[\textbf{Proof of Proposition \ref{prop:1}}]
Let
$$
K:=\esssup_{x \in \Omega}\left(f(x)\right)_{+},
$$
where $s_+ = \max\{ s,0 \}$, 
and suppose $u$ is the weak solution to \eqref{eq:tv_case} with initial value $f$. 
Multipling $\varphi=(u-K)_{+} $ to the both sides of \eqref{eq:entro_con_2}, integrating on $\Omega$, using Green's formula and the equivalent form (a), we have 
$$
\begin{aligned}
\int_{\Omega} \frac{d}{d t}(u-K)_{+}^2 d x &=\int_{\Omega} \operatorname{div}((1-g)\mathbf{z})(u-K)_{+} d x   - \frac{\lambda}{2} \int_{\Omega} \int_{\Omega} \frac{g(x)+g(y)}{2} J(x-y) h(x, y,t) \times\\
& \quad \quad\quad\quad\quad\quad\quad\quad\quad\quad\quad\quad\left((u-K)_{+}(y,t) - (u-K)_{+}(x,t) \right)d y dx \\
&=- \int_{\Omega \cap \{ u > K \}}(1-g)|Du| + \int_{\partial \Omega}(1-g)[\mathbf{z} , \nu] (u-K)_+ d \mathcal{H}^{N-1}\\
& \quad - \frac{\lambda}{2} \int_{\Omega \cap \{ x:u > K \}} \int_{\Omega \cap \{ y: u > K \}} \frac{g(x)+g(y)}{2} J(x-y)  |u(y,t)-u(x,t)|d y dx \\
& \quad - \lambda \int_{\Omega \cap \{ x:u \leqslant K \}} \int_{\Omega \cap \{ y: u > K \}} \frac{g(x)+g(y)}{2} J(x-y) \times\\
&\quad \quad\quad\quad\quad\quad\quad\quad\quad\quad\quad\quad\quad\quad\quad\quad\quad\quad\quad\quad h(x,y,t) (u(y,t)-K)d y dx \\
&\leqslant 0,
\end{aligned}
$$
which implies $u(x, t)-K \leqslant 0$, a.e. for $t \in [0,T]$. 
Similarly, let
$$
L:=- \esssup_{x \in \Omega}\left(-f(x)\right)_{+},
$$
we have  $u(x, t)-L \geqslant 0$, a.e. for $t \in [0,T]$. 

\end{proof}

\begin{proof}[\textbf{Proof of Proposition \ref{prop:3}}]
    Set $\phi = 0 $ in \eqref{eq:entro_con_3}, integrate over $(0,T)$ and let $T \rightarrow \infty$, to obtain 
$$
\int_0^{\infty}\int_{\Omega}(1-g)\left|D u \right| d t \leqslant \frac{1}{2} \int_{\Omega} f^2 dx,  
$$
according to $\|u \|_{L^\infty( \Omega)} \leqslant \|f\|_{L^\infty( \Omega)}$.
Hence, there exists a subsequence $\{t_n\} \rightarrow \infty$ such that $\int_{\Omega} (1-g)|Du(x,t_n)| \rightarrow 0$ as $n \rightarrow \infty$.

Since the conservation of mass of $u$, we deduce 
$$
\|u-\overline{f}\|_{L^2(\Omega)}=\|u -\overline{u }\|_{L^2(\Omega)} \leqslant M \int_\Omega |D u| \leqslant M \int_\Omega (1-g )|D u | , 
$$
by Poincaré's inequality \cite{evans2018measure} when $N=2$.
Similar to the proof of Prop. \ref{prop:1}, it is easily proven that $\int_{\Omega}|u-\overline{f}|^2 d x$ decreases with respect to $t$.
Thus, we have 
\begin{equation*}
   t \left(\int_{\Omega}\left|u-\overline{f}\right|^2 d x\right)^\frac{1}{2} \leqslant \int_0^t \left(\int_{\Omega}\left|u-\overline{f}\right|^2 d x\right)^\frac{1}{2} d s  
   \leqslant M \int_{0}^{\infty} \int_{\Omega} (1-g )|D u | ds 
   \leqslant \frac{M}{2}\int_{\Omega}|f|^2 dx, 
    \end{equation*}
concluding the proof. 
\end{proof}

\subsection{Proof of Theorem \ref{th:nl2local}}
we now
prove the convergence of our model to the local problem. The following lemma 
(see \cite{andreu2008nonlocal}, Proposition 3.2) will be useful for the proof. 
\begin{lemma}\label{lm:lm_ep}
Let $\rho: \mathbb{R}^N \rightarrow \mathbb{R}$ be a nonnegative continuous radial function with compact support, non-identically zero, and $\rho_n(x):=n^N \rho(n x)$. Let $\left\{f_n\right\}$ be a sequence of functions in $L^1(\Omega)$ such that
$$
\int_{\Omega} \int_{\Omega}\left|f_n(y)-f_n(x)\right| \rho_n(y-x) d x d y \leqslant \frac{M}{n}.
$$
If $\left\{f_n\right\}$ is weakly convergent in $L^1(\Omega)$ to $f$, then $ f \in B V(\Omega)$, and 
$$
\rho(z) \chi_{\Omega}\left(x+\frac{1}{n} z\right) \frac{f_n\left(x+\frac{1}{n} z\right)-f_n(x)}{1 / n} \rightharpoonup \rho(z) z \cdot D f,
$$
weakly as measures.
Moreover, if $\rho(x) \geqslant \rho(y)$ for $|x| \leqslant|y|$, then there exists a subsequence $\left\{f_{n_k}\right\}$ such that
$ f_{n_k} \rightarrow f$ in $L^1(\Omega)$ with $f \in B V(\Omega)$.
\end{lemma}

\begin{proof}[\textbf{Proof of Theorem \ref{th:nl2local}}]
Given $\varepsilon >0$, the weak solution to \eqref{eq:TV_nltv_ep} is denoted as $u_\varepsilon$. Then  
there exist $\mathbf{z}_\varepsilon \in  X_2(Q_T)$ with $\|\mathbf{z}_\varepsilon\|_\infty \leqslant 1$ and  $h_\varepsilon \in L^{\infty}\left(\Omega \times \Omega \times (0,T) \right)$  with $\|h_\varepsilon\|_{\infty} \leqslant 1$ such that $h_\varepsilon( x, y,t)=-h_\varepsilon( y, x,t)$,   
\begin{equation}\label{eq:measure_ep}
    \frac{\partial u_\varepsilon}{\partial t}=\operatorname{div}((1-g)\mathbf{z}_\varepsilon)+\lambda \frac{C_{J, 1}}{\varepsilon^{1+N}} \int_{\Omega} \frac{g(x)+g(y)}{2} J\left(\frac{x-y}{\varepsilon}\right) h_\varepsilon(x, y,t) d y, \quad \text{ in } \mathcal{D}^{\prime}\left(Q_T\right).
\end{equation}
For any  $\varphi \in C\left([0, T] ; L^2(\Omega)\right) \cap L^{\infty}(0, T ; B V(\Omega))$, using the nonlocal integrate by parts formula, we have

\begin{equation}\label{eq:weak_ep}
    \begin{aligned}
    &\lambda \frac{C_{J, 1}}{2 \varepsilon^{1+N}}  \int_0^{T} \int_{\Omega} \int_{\Omega} g(x) J\left(\frac{x-y}{\varepsilon}\right)  h_{\varepsilon}(x, y,t)(\varphi(y,t)-\varphi(x,t)) d yd x dt \\
    &+\int_0^{T} \int_{\Omega} \frac{\partial u_\varepsilon}{\partial t} \varphi dxdt+\int_0^{T} \int_{\Omega} \left((1-g )\mathbf{z}_\varepsilon,D\varphi \right) dt=0.
    \end{aligned} 
\end{equation}
Set  $\varphi = u_\varepsilon$ in \eqref{eq:weak_ep} and change variables to obtain 
    \begin{equation*}
    \begin{aligned}
     &  \int_0^{T} \int_{\mathbb{R}^N} \int_{\Omega}\frac{C_{J, 1}}{2} J(z) \chi_{\Omega}\left(x+\varepsilon z\right)\left|\frac{u_{\varepsilon}\left(x+\varepsilon z,t\right)-u_{\varepsilon}(x,t)}{\varepsilon}\right| d x d z dt\\ 
     &+  \int_0^{T} \int_{\Omega} \left(\mathbf{z}_\varepsilon,D u_{\varepsilon} \right) dt +\sup _{0<t < T} \int_{\Omega}\left|u_{\varepsilon}\right|^2 d x \leqslant M.
    \end{aligned}
\end{equation*}

Therefore, by Lemma \ref{lm:lm_ep}, there exist a subsequence $\{\varepsilon_n\} \rightarrow 0$ and $ u \in L^\infty(0,T; B V(\Omega))$,  such that 
\begin{gather*}
    u_{\varepsilon_n} {\rightharpoonup} u, \quad \text { in }  L^2(Q_T),\\
    u_{\varepsilon_n} \rightarrow u, \quad \text { in } L^1(Q_T),
\end{gather*}
and 
\begin{equation}
    \frac{C_{J, 1}}{2} J(z) \chi_{\Omega}\left(x+\varepsilon_n z\right) \frac{u_{\varepsilon_n}\left(x+\varepsilon_n z,t\right)-u_{\varepsilon_n}(x,t)}{\varepsilon_n} \rightharpoonup \frac{C_{J, 1}}{2} J(z) z \cdot D u, 
\end{equation}
weakly as measures. 
Moreover, since the boundedness of $\{\mathbf{z}_{\varepsilon_n}\}$ and $\{h_{\varepsilon_n}\}$, we can also assume that 
\begin{gather}
J(z) \chi_{\Omega}\left(x+\varepsilon_n z\right) h_{\varepsilon_n}\left(x, x+\varepsilon_n z,t\right) \stackrel{*}{\rightharpoonup} \Lambda(x, z,t), \quad \text { in } L^{\infty}\left( \Omega \times \mathbb{R}^N \times (0, T) \right), \label{eq:Jh_ep_weak}\\
\mathbf{z}_{\varepsilon_n}   \stackrel{*}{\rightharpoonup} \mathbf{z} , \quad \text { in } L^\infty\left(Q_T, \mathbb{R}^N\right),  \label{eq:z_ep_weak}
\end{gather}
with $|\Lambda(x, z,t)| \leqslant J(z)$. 

Now, taking $\varepsilon = \varepsilon_n$ in \eqref{eq:weak_ep}, restricting $\varphi$ in $C([0,T], L^2(\Omega) \cap W^{1,1}(\Omega) )$ and changing variables, we have 
\begin{equation}\label{eq:change_var}
    \begin{aligned}
    &\frac{C_{J, 1}}{2} \int_0^{T} \int_{\mathbb{R}^N} \int_{\Omega} g(x) J(z) \chi_{\Omega}\left(x+\varepsilon_n z\right) h_{\varepsilon_n}\left(x, x+\varepsilon_n z,t\right) d z \frac{\varphi\left(x+\varepsilon_n z,t\right)-\varphi(x,t)}{\varepsilon_n} d xdt \\
    &+\int_0^{T} \int_{\Omega} \frac{\partial u_{\varepsilon_n}}{\partial t} \varphi dxdt+\int_0^{T} \int_{\Omega} (1-g )\mathbf{z}_{\varepsilon_n}\cdot \nabla \varphi dx dt =0. 
    \end{aligned} 
\end{equation}
Passing to the limit in \eqref{eq:change_var}, we have 
\begin{equation*}
    \frac{C_{J, 1}}{2} \int_0^{T}  \int_{\mathbb{R}^N } \int_{\Omega} g(x) \Lambda(x, z,t) z \cdot \nabla \varphi d x d z dt + \int_0^{T} \int_{\Omega} \frac{\partial u}{\partial t} \varphi dxdt+\int_0^{T} \int_{\Omega} (1-g )\mathbf{z} \cdot \nabla \varphi  dxdt=0. 
\end{equation*}
Set $\xi=\left(\xi_1, \ldots, \xi_N\right)$, the vector field defined by
\begin{equation*}
    \xi_i  =\frac{C_{J, 1}}{2} g \int_{\mathbb{R}^N} \Lambda(x, z,t) z_i d z + (1-g)\mathbf{z}_i, \quad i=1, \ldots, N .
\end{equation*}
Then $\xi \in L^{\infty}\left(Q_T; \mathbb{R}^N\right)$, and 
$$
\frac{\partial u}{\partial t} = 
\operatorname{div}(\xi), \quad \text { in } \mathcal{D}^{\prime}(Q_T).
$$

Let us see that $\|\xi\|_{\infty} \leqslant 1$. Given $\zeta \in \mathbb{R}^N \backslash\{0\}$, let $R_{\zeta}$ be the rotation such that $\zeta=R_{\zeta}(\mathbf{e}_1|\zeta|)$. After changing variables $z=R_{\zeta}(y)$, we obtain
\begin{equation*}
    \begin{aligned}
\xi  \cdot \zeta &=\frac{C_{J, 1}}{2} g  \int_{\mathbb{R}^N} \Lambda(x, z,t) z \cdot \zeta d z + (1-g )\mathbf{z} \cdot \zeta \\
&=\frac{C_{J, 1}}{2} g  \int_{\mathbb{R}^N} \Lambda\left(x, R_{\zeta}(y),t\right) R_{\zeta}(y) \cdot R_{\zeta}(\mathbf{e}_1|\zeta|) d y  + (1-g )\mathbf{z} \cdot \zeta \\
&=\frac{C_{J, 1}}{2} g  \int_{\mathbb{R}^N} \Lambda\left(x, R_{\zeta}(y),t\right) (y \cdot \mathbf{e}_1) |\zeta| d y  + (1-g )\mathbf{z} \cdot \zeta .
    \end{aligned}
\end{equation*}
Since $J$ is a radial function with $\Lambda(x, z,t) \leqslant J(z)$  and 
$$
C_{J, 1}^{-1}=\frac{1}{2} \int_{\mathbb{R}^N} J(z)\left|z_1\right| d z, 
$$
we obtain 
$$
|\xi  \cdot \zeta| \leqslant \frac{C_{J, 1}}{2} g \int_{\mathbb{R}^N} J(y)\left|y_1\right| d y|\zeta| + (1-g ) |\zeta|, \quad \text { a.e. } (x,t) \in Q_T, 
$$
which implies $\|\xi\|_{\infty} \leqslant 1$.

To finish the proof, we only need to show 
\begin{equation}\label{eq:TV_entropy}
    \int_{\Omega}(u-w) \frac{\partial u}{\partial t} d x 
    \leqslant \int_{\Omega} \xi \cdot \nabla w  dx
    -\int_{\Omega} |Du|, \quad  \text{ a.e. on } [0, T] ,
\end{equation}
for any $w \in L^\infty\left(0, T; W^{1,1}(\Omega) \cap L^2(\Omega) \right)$. 
Taking $\varphi = (w - u_{\varepsilon_n})\eta(t)$  with $\eta(t)\in C[0,T]$, $ \eta(t)\geqslant 0$ in  \eqref{eq:change_var}, and having in mind the equivalent form (a) in Prop. \ref{prop:equivalent}, we have 
\begin{equation*}
    \begin{aligned}
    &\int_0^T \int_{\Omega} \frac{\partial u_{\varepsilon_n}}{\partial t}  (u_{\varepsilon_n}-w )\eta(t)dxdt \\
    &= \frac{C_{J, 1}}{2} \int_0^T  \int_{\mathbb{R}^N} \int_{\Omega} g(x) J(z) \chi_{\Omega}\left(x+\varepsilon_n z\right) h_{\varepsilon_n} \times \\
    &\quad\quad\quad\quad\quad\quad\quad\quad\quad\quad\quad\quad\quad\quad\left(x, x+\varepsilon_n z, t \right) d z \frac{w\left(x+\varepsilon_n z,t\right)-w(x,t)}{\varepsilon_n} \eta(t) d x dt\\
    & \quad  -\frac{C_{J, 1}}{2} \int_0^T  \int_{\mathbb{R}^N} \int_{\Omega} J(z) \chi_{\Omega}\left(x+\varepsilon_n z\right)\left|\frac{u_{\varepsilon_n}\left(x+\varepsilon_n z, t\right)-u_{\varepsilon_n}(x,t)}{\varepsilon_n}\right| \eta(t) dz d xdt \\
    &\quad + \int_0^T  \int_{\Omega} (1-g )\mathbf{z}_{\varepsilon_n} \cdot \nabla w  \eta(t) dx dt-  \int_0^T  \int_{\Omega} \left((1-g )\mathbf{z}_{\varepsilon_n},D u_{\varepsilon_n} \right)\eta(t) dt.
    \end{aligned} 
\end{equation*}
Combining \eqref{eq:Jh_ep_weak}, \eqref{eq:z_ep_weak}, lower semicontinuity in Lemma \ref{weak_continuous_u}, and passing to the limit, we have 
\begin{equation*}
    \begin{aligned}
    &\int_0^T\int_{\Omega} \frac{\partial u}{\partial t} (u - w) \eta(t)dxdt \\
    & \leqslant  \frac{C_{J, 1}}{2}  \int_0^T\int_{\Omega}\int_{\mathbb{R}^N } g(x) \Lambda(x, z,t) z \cdot \nabla w(x,t)\eta(t) d z d x dt + \int_0^T \int_{\Omega} (1-g )\mathbf{z} \cdot \nabla w dx\eta(t) dt \\
    &\quad  -\frac{C_{J, 1}}{2}  \int_0^T \int_{\Omega}\int_{\mathbb{R}^N } g(x)|J(z) z \cdot D u| \eta(t) d z d x dt - \int_0^T \int_{\Omega} (1-g ) |Du| \eta(t) dt\\
    &=\int_0^T \int_{\Omega} \xi \cdot \nabla w \eta(t) d x dt - \int_0^T \int_{\Omega} (1-g ) |Du| \eta(t) dt\\
    &\quad -\frac{C_{J, 1}}{2} \int_0^T \int_{\Omega} \int_{\mathbb{R}^N}g(x) |J(z) z \cdot D u| \eta(t) d z d x dt.
    \end{aligned}
\end{equation*}
Now, for every $x \in \Omega$ such that the Radon-Nikodym derivative $\frac{D u}{|D u|}(x) \neq 0$, we have
$$
\begin{aligned}
&\frac{C_{J, 1}}{2} \int_0^T  \int_{\Omega} \int_{\mathbb{R}^N} g(x) |J(z) z \cdot D u| \eta(t) dzdx dt\\
& =\frac{C_{J, 1}}{2} \int_0^T  \int_{\Omega} \int_{\mathbb{R}^N} g(x) \eta(t) J(z)\left|z \cdot \frac{D u}{|D u|}(x)\right| d z d|D u|(x)  dt.
\end{aligned}
$$
Let $R_x$ be the rotation such that $R_x\left(\frac{D u}{|D u|}(x)\right)=\mathbf{e}_1\left|\frac{D u}{|D u|}(x)\right|$. Then, since $J$ is a radial function and $\left|\frac{D u}{|D u|}(x)\right|=1$, $|D u|$-a.e. in $\Omega$, if we make the change of variables $y=R_x(z)$, we have
$$
\begin{aligned}
&\frac{C_{J, 1}}{2} \int_0^T  \int_{\Omega} \int_{\mathbb{R}^N} g(x) |J(z) z \cdot D u| \eta(t) dzdx dt\\
& =\frac{C_{J, 1}}{2} \int_0^T  \int_{\Omega} \int_{\mathbb{R}^N} g(x) \eta(t) J(z)\left|R_x(z) \cdot R_x \left(\frac{D u}{|D u|}(x)\right)\right| d z d|D u|(x)  dt\\
& =\frac{C_{J, 1}}{2} \int_0^T  \int_{\Omega} \int_{\mathbb{R}^N} g(x) \eta(t) J(y)\left|y_1\right| d y d|D u|(x)  dt\\
&= \int_0^T  \int_{\Omega} \eta(t)g(x)|D u|  dt.
\end{aligned}
$$
Consequently, because of the arbitrariness of $\eta(t)$, \eqref{eq:TV_entropy} holds.

\end{proof}

\section{Experiments}\label{sec:experi}
In this section, finite difference method is used to discretize problem \eqref{eq:tv_case}. The AA model \cite{aubert2008variational}, F1P-AA model \cite{Tianling2022DCDS} and LNL-2 model \cite{shi2021coupling} are selected for comparison.

\subsection{Numerical scheme}
Equation \eqref{eq:tv_case} comprises the local and nonlocal TV components.
For these two distinct components, the midpoint discretization scheme is applied to the local TV term, while the nonlocal operators in \cite{lou2010image} are employed to discretize the nonlocal component.

Denoting $u_{i, j}^{(n)}$ as the value of $u$ at pixel $(i, j)$ ($1\leqslant i \leqslant I$, $1\leqslant j \leqslant J $) on the $n$-th iteration, and adding $1-\lambda$ to the local part for easier adjustment, then \eqref{eq:tv_case} can be discretized as 
\begin{equation}\label{eq:dis_ourmodel}
    \begin{aligned}
    \frac{u_{i, j}^{(n+1)}-u_{i, j}^{(n)}}{\tau}=&(1-\lambda)\left(C_{i+\frac{1}{2}, j}^{(n)}\left(u_{i+1, j}^{(n)}-u_{i, j}^{(n)}\right)+C_{i-\frac{1}{2}, j}^{(n)}\left(u_{i-1, j}^{(n)}-u_{i, j}^{(n)}\right)\right. \\
    &\left.+C_{i, j+\frac{1}{2}}^{(n)}\left(u_{i, j+1}^{(n)}-u_{i, j}^{(n)}\right) +C_{i, j-\frac{1}{2}}^{(n)}\left(u_{i, j-1}^{(n)}-u_{i, j}^{(n)}\right)\right)\\
    &+\lambda g_{i, j} \sum_{(s, t) \in \mathcal{N}_{i, j}}  \alpha_{i, j}^{s, t}\left(u_{s, t}^{(n)}-u_{i, j}^{(n)}\right),
    \end{aligned}
    \end{equation}
where 
\begin{equation*}
    \begin{aligned}
    C_{i+\frac{1}{2}, j} =&\left(1-\frac{g_{i+1, j}+g_{i, j}}{2}\right) \times \\
    &\left(\left(u_{i+1, j}^{(n)}-u_{i, j}^{(n)}\right)^2+ \left(\frac{u_{i+1, j+1}^{(n)}+u_{i, j+1}^{(n)}-u_{i+1, j-1}^{(n)}-u_{i, j-1}^{(n)}}{4}\right)^2+\epsilon\right)^{-\frac{1}{2}} .
    \end{aligned}
    \end{equation*}
    A similar approach is employed to define other local coefficients at midpoints.	
    $\epsilon$ is a small positive constant to prevent division by zero in the denominator.	

For the nonlocal component, we first denote $\mathcal{N}_{i, j}$ as the set of neighbors of point $(i, j)$, $h$ as a positive constant that quantifies the noise level, and ${\left\|f\left(\mathcal{N}_{i, j}\right)-f\left(\mathcal{N}_{s, t}\right)\right\|_{2, a}^2}$ represents the Gaussian-weighted Euclidean difference.	
Following the numerical scheme of nonlocal TV proposed in \cite{lou2010image}, the nonlocal coefficient $\alpha_{i, j}^{s, t}$ is computed as	
\begin{equation*}
    \alpha_{i, j}^{s, t} = w_{i, j}^{s, t} \left|  \left(\sum_{k,m \in \mathcal{N}_{i, j}}  w_{i, j}^{k,m} (u_{i,j}-u_{k,m})^2 \right)^{-\frac{1}{2}} +  \left(\sum_{k,m \in \mathcal{N}_{s, t}}  w_{s, t}^{k,m} (u_{s,t}-u_{k,m})^2 \right)^{-\frac{1}{2}}\right| , 
\end{equation*}
where $w_{i, j}^{s, t}$ is designed to measure the similarity between the two patches $\mathcal{N}_{i, j}$ and $\mathcal{N}_{s, t}$:
\begin{equation}\label{eq:w_choice}
    w_{i, j}^{s, t}=\frac{1}{Z_{i , j}} e^{-\frac{\left\|f\left(\mathcal{N}_{i, j}\right)-f\left(\mathcal{N}_{s, t}\right)\right\|_{2, a}^2}{h^2}},
\end{equation}
and $Z_{i, j}$ is a constant such that $\sum_{(s, t) \in \mathcal{N}_{i, j}} w_{i, j}^{s, t}=1$. 

Given that multiplicative noise exhibits spatial correlation, an appropriate grayscale indicator is essential for enhancing the denoising performance.	
Since higher grayscale values are more susceptible to degradation in an image, we choose 	
\begin{equation}\label{eq:G_choice}
    \widetilde{f}_{i,j}=\left(\frac{\left(f_\sigma\right)_{i,j}}{\max_{1\leqslant i \leqslant I , 1\leqslant j \leqslant J}\left(f_\sigma\right)_{i,j}}\right)^\beta,
\end{equation}
where $\beta$ is a nonnegative constant, and $f_\sigma$ denotes the convolution of $f$ with the Gaussian kernel $G_\sigma$, which has a standard deviation of $\sigma$.	
A higher grayscale value corresponds to a larger $\widetilde{f}_{i,j}$, thereby increasing the diffusion velocity.	
Then, the grayscale indicator $g_{i,j}$ is set to 
\begin{equation*}
g_{i,j}= \begin{cases*} \widetilde{f}_{i,j}, \quad \text{if } \lambda \geqslant 0.5,\\
    1- \widetilde{f}_{i,j}, \quad \text{if } \lambda < 0.5.
\end{cases*}
\end{equation*}
As $\lambda$ varies, the relative dominance of nonlocal and local TV terms also changes. 
When $\lambda \geqslant 0.5$, the nonlocal term dominates the diffusion process; thus, multiplying $\widetilde{f}_{i,j}$ by the nonlocal term is more appropriate, and vice versa.

The Neumann boundary condition can be efficiently implemented by
\begin{equation*}
    u_{i, 0}^n=u_{i, 1}^n, \quad u_{0, j}^n=u_{1, j}^n, \quad
    u_{I, j}^n=u_{I+1, j}^n, \quad u_{i, J}^n=u_{i, J+1}^n.
\end{equation*}

In parameter selection, most parameters remain unchanged, while some vary depending on the experimental setup.	
When computing the weight $w_{i, j}^{s, t}$, $h$ is estimated using the algorithm in \cite{immerkaer1996fast}. The search window size is set to $21 \times 21$, and the patch size to $10\times 10$, with only the 20 most similar neighbors of each pixel are considered for weighting. 	
The fixed parameters $\tau$, $\epsilon$ are 0.2 and $10^{-5}$, respectively.

\subsection{Experimental results}
In this subsection, we present two distinct experiments.	
`Experiment 1' aims to demonstrate the significance of coupling between nonlocal TV and local TV.	
`Experiment 2' is a comparative experiment  designed to evaluate the denoising performance of our model against other multiplicative denoising models.	
We select five images as the test set, with  `Hybrid' utilized in `Experiment 1', and the remaining images used in `Experiment 2'. 

\begin{figure}[htbp]
    \centering  
    \subfigure[Hybrid]{
        \label{Hybrid}
        \includegraphics[width=0.18\textwidth]{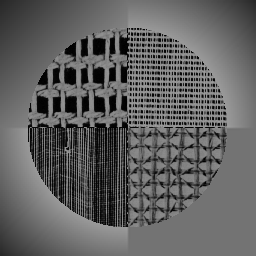}}
    \subfigure[Barbara]{
        \label{Barbara}
        \includegraphics[width=0.18\textwidth]{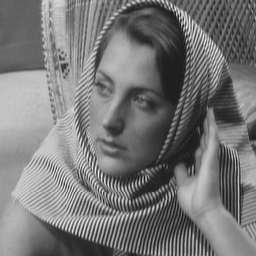}}
    \subfigure[Tower]{
        \label{Tower}
        \includegraphics[width=0.18\textwidth]{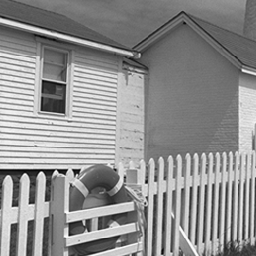}}
    \subfigure[Texture]{
        \label{Texture}
        \includegraphics[width=0.18\textwidth]{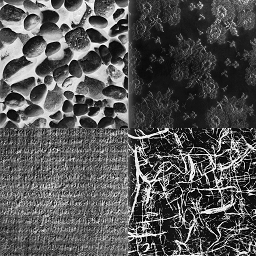}}
    \subfigure[Satelite]{
        \label{Satelite}
        \includegraphics[width=0.18\textwidth]{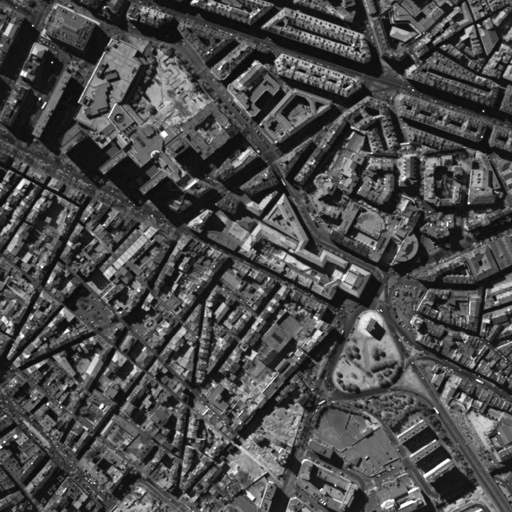}}
    \caption{Test images used in experiments.}
    \label{five_ima}
\end{figure}

\begin{table}[]
    \setlength{\tabcolsep}{10pt}
    \centering
    \caption{The comparison of PSNR and SSIM on test images.}\label{simlua}
    \begin{tabular}{cccccc}
    \hline
                        Noise Level  &   Name       & AA & F1P-AA & LNL-2 & Ours \\ \hline
    \multirow{4}{*}{$L=4$}  & Barbara   &  20.95/0.518  &   20.95/0.318      &  20.94/\textbf{0.579}     &   \textbf{21.09}/0.563   \\
                          & Tower    & 19.30/0.573   &  19.19/0.375       &   \textbf{19.99}/\textbf{0.658}    &  19.62/0.594    \\
                          & Texture  & 17.87/0.691   &  17.67/0.643       &  17.93/0.698     &   \textbf{17.98}/\textbf{0.700}   \\
                          & Satelite &  22.19/0.857  &   21.62/0.594      &  20.89/0.854     &   \textbf{22.22}/\textbf{0.859}   \\ \hline
    \multirow{4}{*}{$L=10$} & Barbara   & 22.60/0.597  &   22.38/0.427      &   \textbf{23.93}/0.679    &   23.79/\textbf{0.700}   \\
                          & Tower    & 22.05/0.703   &   21.76/0.498      &   23.18/0.760    &   \textbf{23.35}/\textbf{0.777}   \\
                          & Texture  & 20.34/0.814  &  20.18/0.782       &  20.57/0.819     &   \textbf{20.62}/\textbf{0.820}   \\
                          & Satelite & 24.21/0.923   &   23.65/0.721      &  23.17/0
                          922     &  \textbf{24.32}/\textbf{0.926}    \\ \hline
    \end{tabular}
    \end{table}

In `Experiment 1', we want to highlight the importance of coupling.	
We compare our model \eqref{eq:tv_case} with the local adaptive TV flow and the adaptive NLTV flow, which correspond to $\lambda = 1$ and $\lambda=0$ in \eqref{eq:dis_ourmodel}, respectively.	
The noisy image is generated by multiplying Gamma noise with $L=10$ onto the `Hybrid' image.	
The chosen parameters are $\sigma=3$, $\beta=1$ for adaptive NLTV, $\sigma=2$, $\beta=2$ for adaptive TV, and $\sigma=3$, $\beta=2$ for our model, respectively.	
Specifically, we let the grayscale indicator $g_{i,j}$ as 
\begin{equation*}
    g_{i,j}= \begin{cases*} (G_{\sigma_1}*\chi)_{i,j}\widetilde{f}_{i,j}, \quad \text{if } \lambda \geqslant 0.5,\\
        1- (G_{\sigma_1}*\chi)_{i,j} \widetilde{f}_{i,j}, \quad \text{if } \lambda < 0.5,
    \end{cases*}
\end{equation*}
where $\chi(x) = 0.999$ for texture regions and $\chi(x) = 0.001$ for background regions. Here, $G_{\sigma_1}*\chi$ represents the convolution of $\chi$ with a Gaussian kernel $G_{\sigma_1}$, where the standard deviation is set to $\sigma_1 = 0.5$.	
Under these conditions, the restoration results are presented in Figure \ref{Exp11}.	
As illustrated in Figure \ref{Exp11}, the adaptive TV model achieves superior denoising performance in homogeneous regions compared to the nonlocal adaptive TV model.	
The nonlocal adaptive TV model leaves residual noise in homogeneous regions but effectively preserves textures.	
In contrast, the coupling model \eqref{eq:tv_case} mitigates the drawbacks of both approaches, achieving higher PSNR and SSIM values.

\begin{figure}[htbp]
    \centering
    \begin{tabular}{@{}c@{~}c@{~}c@{~}c@{}}
        \includegraphics[width=0.23\textwidth]{fig/hybrid.png} &
        \includegraphics[width=0.23\textwidth]{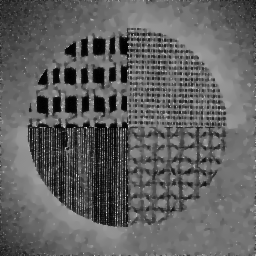} &
        \includegraphics[width=0.23\textwidth]{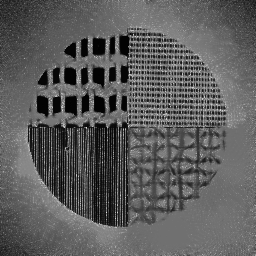}&
        \includegraphics[width=0.23\textwidth]{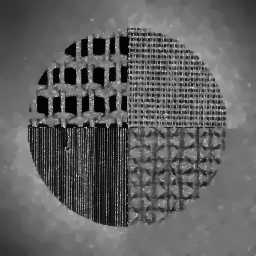} \\

        \footnotesize{Clean Image} & 
        \footnotesize{PSNR:24.22/SSIM:0.81}  & 
        \footnotesize{PSNR:23.82/SSIM:0.73}  & 
        \footnotesize{PSNR:24.81/SSIM:0.88} \\

        \includegraphics[width=0.23\textwidth]{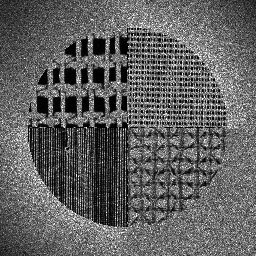} &
        \includegraphics[width=0.23\textwidth]{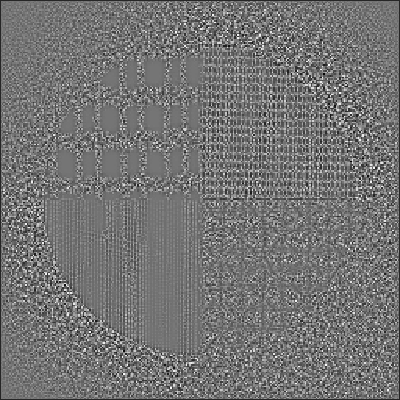}&
        \includegraphics[width=0.23\textwidth]{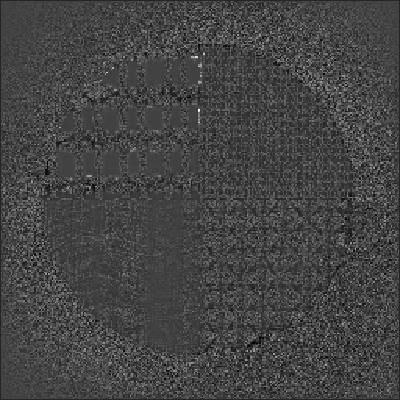} &
        \includegraphics[width=0.23\textwidth]{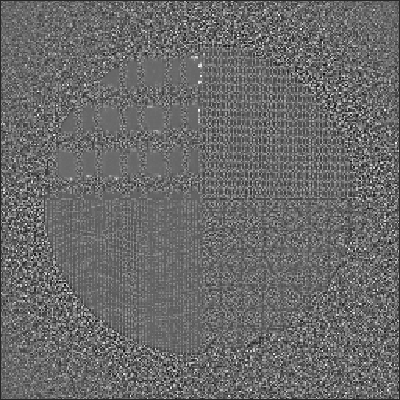}  \\

        \footnotesize{Noisy Image} & 
        \footnotesize{Adaptive TV ($\lambda=0$)}  & 
        \footnotesize{Adaptive NLTV ($\lambda=1$)}  & 
        \footnotesize{Ours ($\lambda=0.3$)} \\

      \end{tabular}
      \caption{The restoration results $u$, difference images ($f - u$), PSNR values, and SSIM values for a synthesis
      image under different $\lambda$ in Eq.\eqref{eq:tv_case}.  }
      \label{Exp11}
\end{figure}

\begin{table}[]
    \centering
    \caption{The choice of parameters in `Experiment 2'.}\label{para_choice}
    \begin{tabular}{cccccc}
    \hline
                            Noise Level&  Parameters           & Barbara     & Tower     & Texture   & Satalite \\ \hline
    \multirow{3}{*}{$L=4$}  & $\lambda$   &   0.8          &   0.8        &  0.9         & 0.7         \\
                            & $\sigma$    &    3         &     3      &     3      &    5      \\
                            & $\beta$     &    3         &     2      &      2     &    3      \\ \hline
    \multirow{3}{*}{$L=10$} & $\lambda$   &  0.9           &  0.8         & 0.9           &  0.9        \\
                            & $\sigma$    &   2          &   3        &  3         &    5      \\
                            & $\beta$     &   2          &    2       &   2        &    3      \\ \hline
    \end{tabular}
    \end{table}

In `Experiment 2', the noisy images are generated by multiplying the original images with Gamma noise at noise levels $L=4$ and $L=10$.	
The other three  models: AA model \cite{aubert2008variational}, F1P-AA model \cite{Tianling2022DCDS} and LNL-2 model \cite{shi2021coupling} are selected for comparison. 
Among these, the AA model employs local TV regularization, while the F1P-AA model is a fractional nonlocal TV approach derived from the AA model.	
The LNL-2 model is a local and nonlocal coupled approach, corresponding to Equation (5) in \cite{shi2021coupling}.
Table \ref{simlua} presents the PSNR and SSIM values for different noise levels $L$, while Table \ref{para_choice} outlines the parameter choices for our model, corresponding to the results in Table \ref{simlua}.	
The restoration results for noise levels $L=10$ and $L=4$ are depicted in Figures \ref{test-L-10} and \ref{test-L-4}, respectively.	

Figures \ref{test-L-10} and \ref{test-L-4} demonstrate that the LNL-2 model and our model preserve texture more effectively than the AA and F1P-AA models, confirming the effectiveness of coupling methods. 
The LNL-2 model exhibits superior visual quality under higher noise levels. However, our model  removes background noise more effectively.		

As shown in Table \ref{simlua}, our model achieves the highest PSNR and SSIM in most cases when $L=10$.	
For $L=4$, the LNL-2 model demonstrates strong denoising performance on the `Tower' image.	
Furthermore, our model exhibits performance comparable to the LNL-2 and AA models on other images, with slightly superior results compared to the remaining models.

\begin{figure}[htbp]
    \centering
    \begin{tabular}{@{}c@{~}c@{~}c@{~}c@{~}c@{}}
        \includegraphics[width=0.19\textwidth]{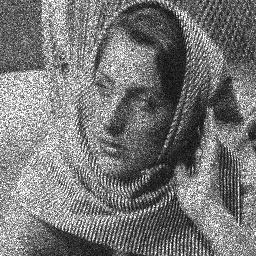} &
        \includegraphics[width=0.19\textwidth]{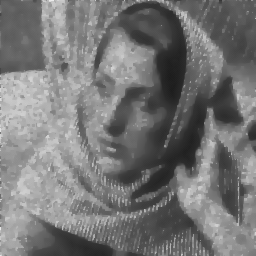} &
        \includegraphics[width=0.19\textwidth]{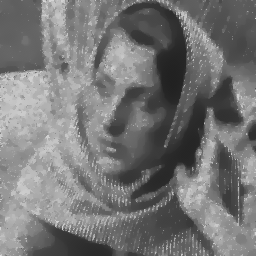}&
        \includegraphics[width=0.19\textwidth]{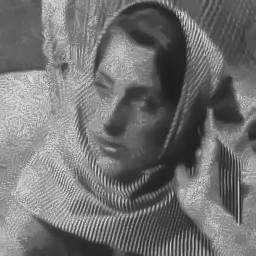} &
        \includegraphics[width=0.19\textwidth]{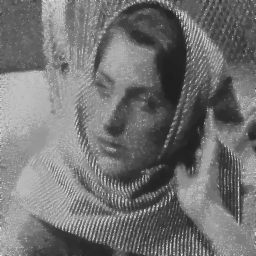}\\

        \includegraphics[width=0.19\textwidth]{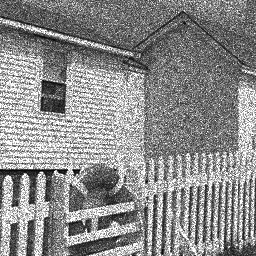}&
        \includegraphics[width=0.19\textwidth]{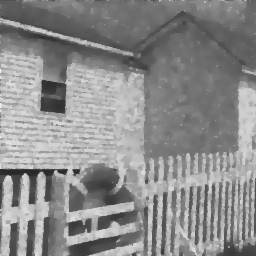} &
        \includegraphics[width=0.19\textwidth]{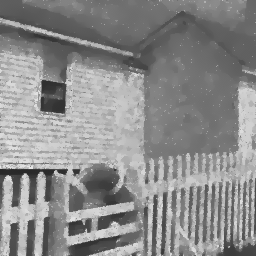}&
        \includegraphics[width=0.19\textwidth]{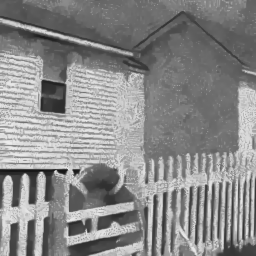} &
        \includegraphics[width=0.19\textwidth]{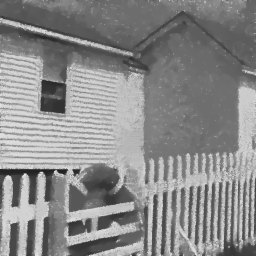}  \\

        \includegraphics[width=0.19\textwidth]{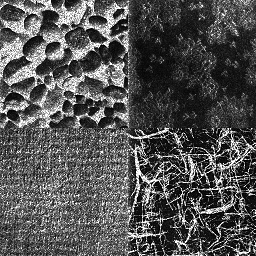}&
        \includegraphics[width=0.19\textwidth]{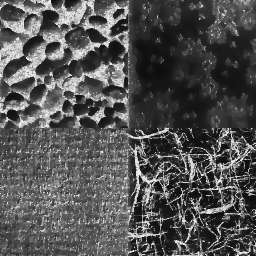} &
        \includegraphics[width=0.19\textwidth]{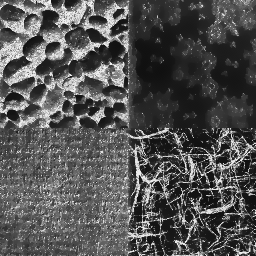}&
        \includegraphics[width=0.19\textwidth]{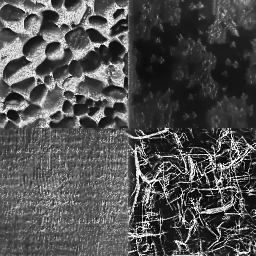} &
        \includegraphics[width=0.19\textwidth]{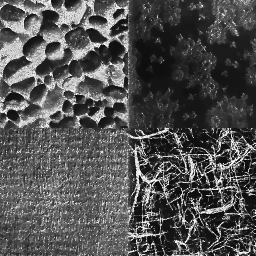}  \\

        \includegraphics[width=0.19\textwidth]{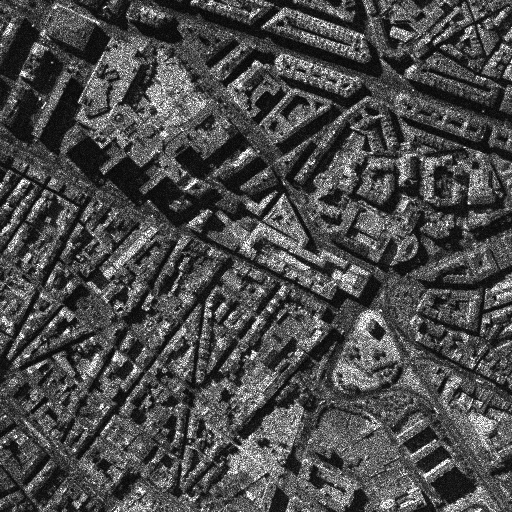}&
        \includegraphics[width=0.19\textwidth]{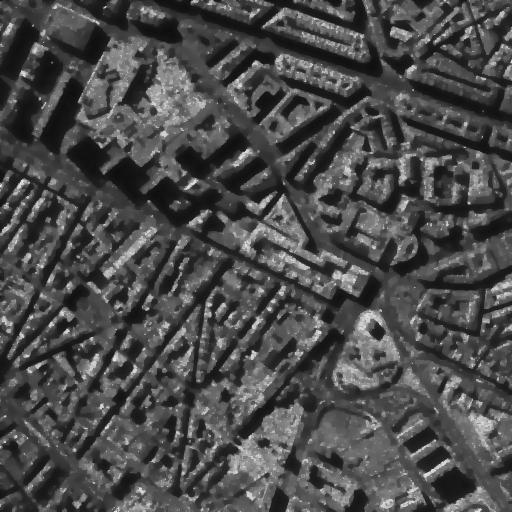} &
        \includegraphics[width=0.19\textwidth]{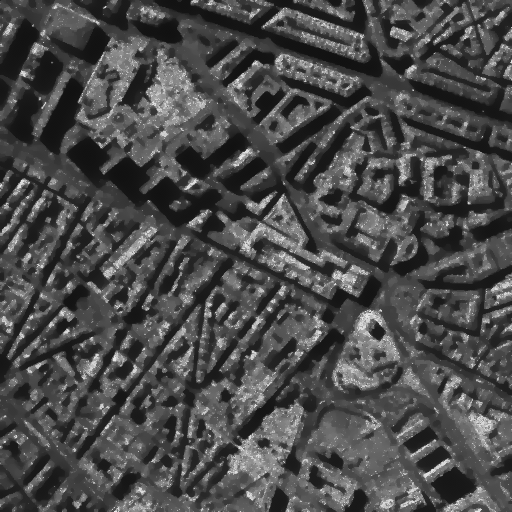}&
        \includegraphics[width=0.19\textwidth]{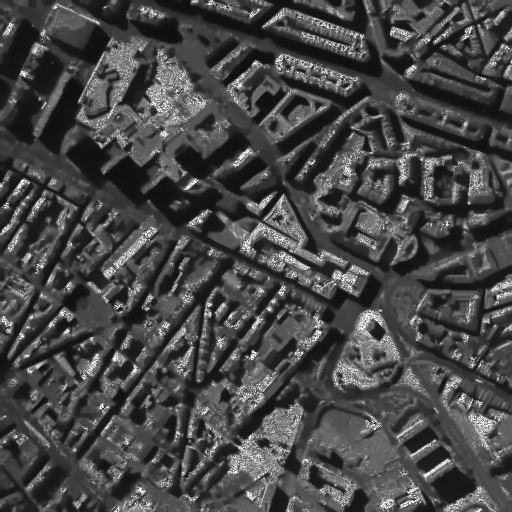} &
        \includegraphics[width=0.19\textwidth]{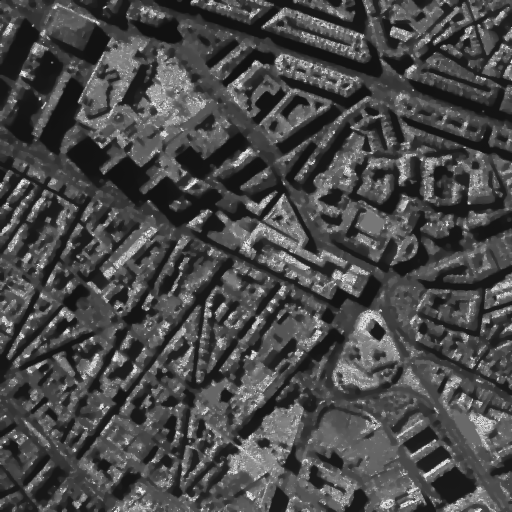}  \\

        \footnotesize{Noisy Image} & 
        \footnotesize{AA model}  & 
        \footnotesize{F1P-AA model}  & 
        \footnotesize{LNL-2 model}  & 
        \footnotesize{Ours} \\

      \end{tabular}
      \caption{Restoration results of four methods for test images with noise level $L=10$.}
      \label{test-L-10}
\end{figure}

\begin{figure}[htbp]
    \centering
    \begin{tabular}{@{}c@{~}c@{~}c@{~}c@{~}c@{}}
        \includegraphics[width=0.19\textwidth]{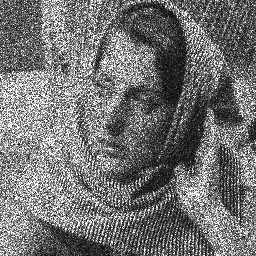} &
        \includegraphics[width=0.19\textwidth]{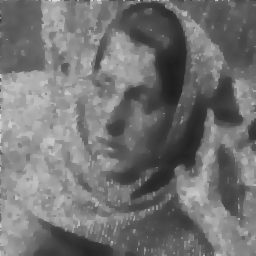} &
        \includegraphics[width=0.19\textwidth]{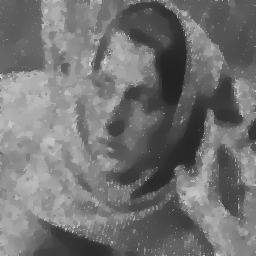}&
        \includegraphics[width=0.19\textwidth]{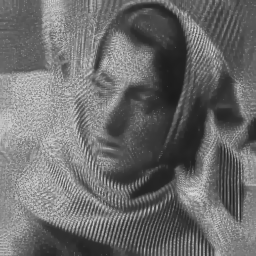} &
        \includegraphics[width=0.19\textwidth]{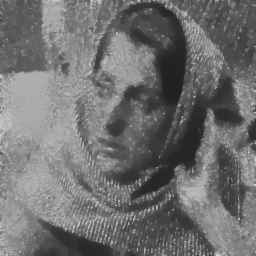}\\

        \includegraphics[width=0.19\textwidth]{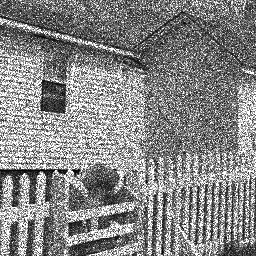}&
        \includegraphics[width=0.19\textwidth]{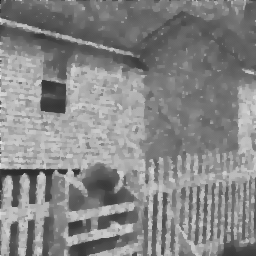} &
        \includegraphics[width=0.19\textwidth]{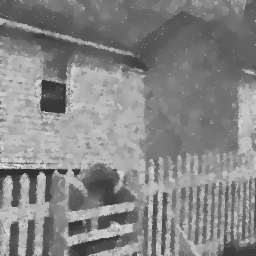}&
        \includegraphics[width=0.19\textwidth]{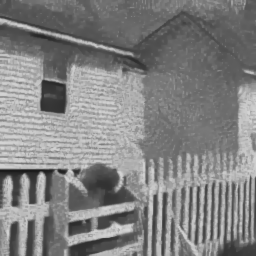} &
        \includegraphics[width=0.19\textwidth]{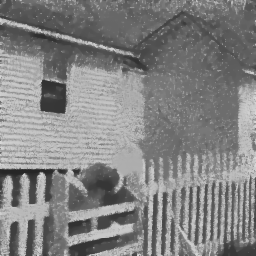}  \\

        \includegraphics[width=0.19\textwidth]{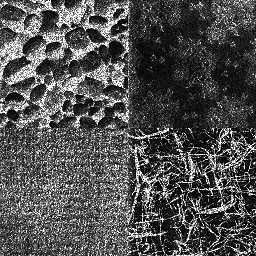}&
        \includegraphics[width=0.19\textwidth]{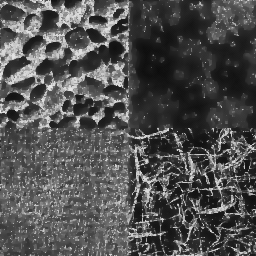} &
        \includegraphics[width=0.19\textwidth]{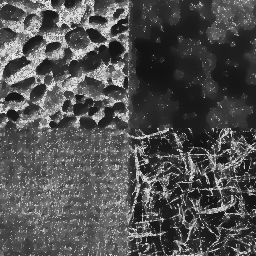}&
        \includegraphics[width=0.19\textwidth]{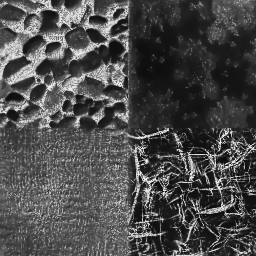} &
        \includegraphics[width=0.19\textwidth]{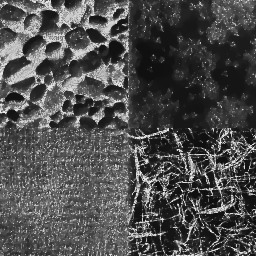}  \\

        \includegraphics[width=0.19\textwidth]{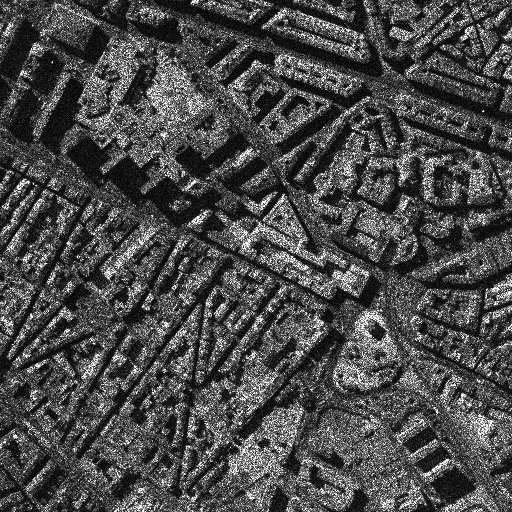}&
        \includegraphics[width=0.19\textwidth]{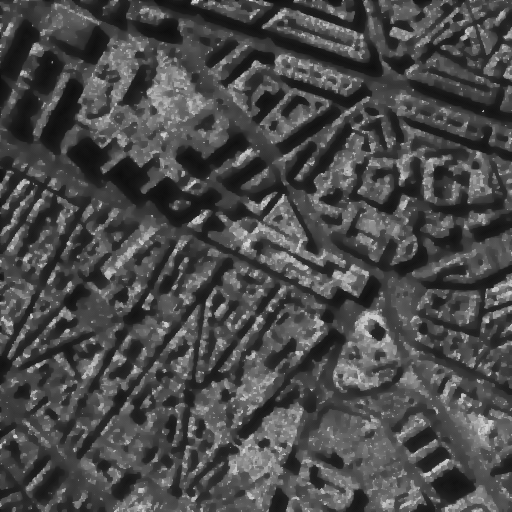} &
        \includegraphics[width=0.19\textwidth]{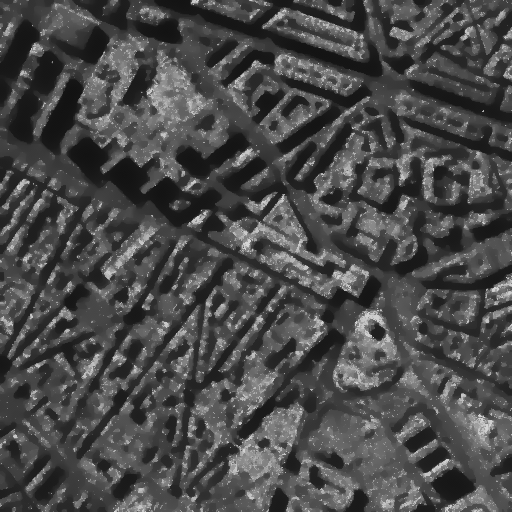}&
        \includegraphics[width=0.19\textwidth]{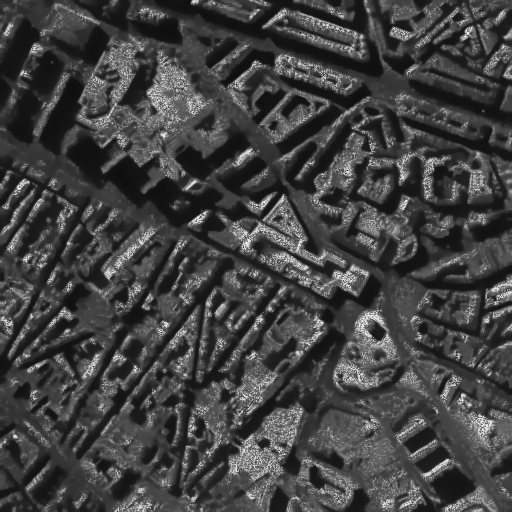} &
        \includegraphics[width=0.19\textwidth]{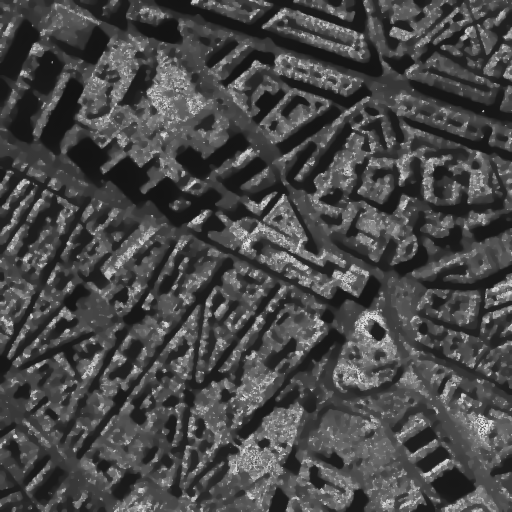}  \\

        \footnotesize{Noisy Image} & 
        \footnotesize{AA model}  & 
        \footnotesize{F1P-AA model}  & 
        \footnotesize{LNL-2 model}  & 
        \footnotesize{Ours} \\

      \end{tabular}
      \caption{Restoration results of four methods for test images with noise level $L=4$. }
      \label{test-L-4}
\end{figure}

\section{Conclusion}
To address the limitations of the local TV model, which fails to preserve texture, and the nonlocal TV model, which exhibits weak regularization, we propose a local-nonlocal coupled total variation flow.	
We analyze the existence, uniqueness, and equivalent formulations of the weak solution to our model when the initial noise image $f \in BV(\Omega) \cap L^2(\Omega)$. 
The maximum principle and asymptotic behavior are established under additional assumptions.	
Notably, the weak solutions of our model converge to the TV model under rescaling.	
Numerical experiments are conducted to demonstrate the significance of coupling and to illustrate the better performance of the proposed model compared to three other models.


\section*{Acknowledgments}

\section*{Declaration of competing interest }
The authors declare that they have no known competing financial interests or
personal relationships that could have appeared to influence the work reported in this paper.

\section*{Data available}
Data will be made available on request.

\bibliography{Coupl_tv_nltv}


\end{document}